\RequirePackage[l2tabu, orthodox]{nag}

\documentclass[
11pt,                          % standard font size
english                        % standard language
]{article}

\usepackage[english]{babel}    % with explicit language
\usepackage{amsmath}           % ams mathematical stuff

\usepackage[utf8]{inputenc}    % smart input of funny chars
\usepackage[T1]{fontenc}       % also for the font encoding
\usepackage{longtable}         % tables longer than one page
\usepackage{exscale}           % large summation signs in 11pt
\usepackage[final]{graphicx}   % to include pdf pictures
\usepackage[sort]{cite}        % nicer citations
\usepackage{array}             % nice tables
\usepackage{wasysym}           % smiley symbols
\usepackage[a4paper]{geometry} % geometry of page layout
\usepackage{xspace}            % better spacing after macros
\usepackage{tikz}              % for commutative diagrams and stuff
\usepackage{ifdraft}           % to determine whether draft mode
\usepackage{amssymb}           % ams symbols
\usepackage{bbm}               % nicer bbm fonts
\usepackage{mathtools}         % some nice tricks for stackrel
\usepackage[amsmath,thmmarks,hyperref]{ntheorem} % nicer theorems
\usepackage{mathrsfs}          % nice calligraphic font
\usepackage{stmaryrd}          % more brackets
\usepackage{paralist}          % better enumerate lists
\usepackage[expansion=false    % no font expansion
           ]{microtype}        % only protrusion
\usepackage[nottoc]{tocbibind} % refs and index in the toc
\usepackage[backref=page,      % backrefs in the bibliography
           final=true,         % always treat as final
           pdfpagelabels       % use pdf page labels
           ]{hyperref}         % hyperrefs are cool!

\newcommand{\Sym}                     {\mathrm{S}}
\newcommand{\lie}[1]                  {\mathfrak{#1}}
\newcommand{\algebra}[1]      {\mathscr{#1}}
\DeclareMathOperator{\Pol}    {\mathrm{Pol}}
\newcommand{\I}              {\mathrm{i}}
\newcommand{\E}              {\mathrm{e}}
\newcommand{\Tensor}                  {\mathrm{T}}
\newcommand{\tensor}[1][{}]           {\mathbin{\otimes_{\scriptscriptstyle{#1}}}}
\DeclareMathOperator{\Symmetrizer}    {\mathscr{S}}
\DeclareMathOperator{\image} {\mathrm{im}}
\newcommand{\id}             {\mathsf{id}}
\DeclareMathOperator{\ad}     {\mathrm{ad}}
\newcommand{\argument}       {\,\cdot\,}

\newcommand{\At}[1]          {\Big|_{#1}}
\newcommand{\Bounded}[1][{}]   {\mathfrak{B}_{\scriptscriptstyle{#1}}}
\newcommand{\Unit}           {\mathbb{1}}

\newcommand{\ostar}{\star}

\newcommand{\ocoproduct}{\Delta}

\newcommand{\Eta}{\mathrm{H}}

\newcommand{\bch}[2]{\mathrm{BCH}\left(#1, #2\right)}
\newcommand{\bchpart}[3]{\mathrm{BCH}_{#1}\left(#2, #3\right)}
\newcommand{\bchparts}[4]{\mathrm{BCH}_{#1, #2}\left(#3, #4\right)}
\newcommand{\bchtilde}[4]{\widetilde{\mathrm{BCH}}_{#1, #2}\left(#3; #4\right)}
\newcommand\ot[2]{\stackrel{\mathclap{#1}}{#2}}
\newcommand{\refitem}[1] {\textit{\ref{#1}.)}}

\theoremheaderfont{\normalfont\bfseries}
\theorembodyfont{\itshape}

\newtheorem{lemma}{Lemma}[section]
\newtheorem{proposition}[lemma]{Proposition}
\newtheorem{theorem}[lemma]{Theorem}
\newtheorem{corollary}[lemma]{Corollary}
\newtheorem{definition}[lemma]{Definition}

\theorembodyfont{\rmfamily}

\newtheorem{example}[lemma]{Example}
\newtheorem{remark}[lemma]{Remark}
\newtheorem{question}[lemma]{Question}

\theorembodyfont{\itshape}
\theoremnumbering{Roman}
\newtheorem{maintheorem}{Main Theorem}

\newenvironment{lemmalist}{\begin{compactenum}[\itshape i.)]}{\end{compactenum}}
\newenvironment{theoremlist}{\begin{compactenum}[\itshape i.)]}{\end{compactenum}}
\newenvironment{propositionlist}{\begin{compactenum}[\itshape i.)]}{\end{compactenum}}
\newenvironment{definitionlist}{\begin{compactenum}[\itshape i.)]}{\end{compactenum}}
\newenvironment{corollarylist}{\begin{compactenum}[\itshape i.)]}{\end{compactenum}}

\theoremheaderfont{\scshape}
\theorembodyfont{\normalfont}
\theoremstyle{nonumberplain}
\theoremseparator{:}
\theoremsymbol{\hbox{$\boxempty$}}
\newtheorem{proof}{Proof}
\theoremsymbol{\hbox{$\triangledown$}}
\newtheorem{subproof}{Proof}

\graphicspath{{../tikz/}}

\usetikzlibrary{matrix}
\usetikzlibrary{arrows}
\usetikzlibrary{patterns}
\usetikzlibrary{decorations.pathreplacing}

\title{Convergence of the Gutt Star Product}

\author{
  \textbf{Chiara Esposito$^1$}\thanks{\texttt{chiara.esposito@mathematik.uni-wuerzburg.de}},
  \addtocounter{footnote}{2}
  \textbf{Paul Stapor$^2$}\thanks{\texttt{paul.stapor@helmholtz-muenchen.de}},
  \addtocounter{footnote}{2}
  \textbf{Stefan Waldmann$^1$}\thanks{\texttt{stefan.waldmann@mathematik.uni-wuerzburg.de}}\\[0.5cm]
  \begin{minipage}{6cm}
      \centering
      ${}^1$Institut für Mathematik \\
      Lehrstuhl für Mathematik X \\
      Universität Würzburg \\
      Campus Hubland Nord \\
      Emil-Fischer-Straße 31 \\
      97074 Würzburg \\
      Germany
  \end{minipage}
  \begin{minipage}{8cm}
      \centering
      ${}^2$Helmholtz Zentrum München \\
      Deutsches Forschungszentrum für \\
      Gesundheit und Umwelt (GmbH) \\
      Institute of Computational Biology \\
      Ingolstädter Landstr. 1 \\
      85764 Neuherberg \\
      Germany
  \end{minipage}
}

\date{September 2015}

\begin{document}

%
% title page
%

\maketitle

%
% abstract
%

\begin{abstract}
    In this work we consider the Gutt star product viewed as an
    associative deformation of the symmetric algebra
    $\Sym^\bullet(\lie{g})$ over a Lie algebra $\lie{g}$ and discuss
    its continuity properties: we establish a locally convex topology
    on $\Sym^\bullet(\lie{g})$ such that the Gutt star product becomes
    continuous. Here we have to assume a mild technical condition on
    $\lie{g}$: it has to be an \emph{Asymptotic Estimate} Lie
    algebra. This condition is e.g. fulfilled automatically for all
    finite-dimensional Lie algebras.  The resulting completion of the
    symmetric algebra can be described explicitly and yields not only
    a locally convex algebra but also the Hopf algebra structure maps
    inherited from the universal enveloping algebra are continuous.
    We show that all Hopf algebra structure maps depend analytically
    on the deformation parameter. The construction enjoys good
    functorial properties.
\end{abstract}

\newpage

%
% table of contents
%

\tableofcontents
\newpage

%
% Introduction
%

\section{Introduction}
\label{sec:Introduction}

Formal deformation quantization as introduced in
\cite{bayen.et.al:1978a} has reached a remarkable state where the
existence and the classification of formal star products is by now
understood very well: Kontsevich's formality theorem gives both, the
general existence of formal star products on Poisson manifolds as well
as their classification up to equivalence. In the symplectic case,
earlier results gave the existence
\cite{dewilde.lecomte:1983b,fedosov:1994a} as well as the
classification \cite{nest.tsygan:1995a, bertelson.cahen.gutt:1997a,
  deligne:1995a}.

Beside the symplectic situation, the linear Poisson structures,
i.e. the Kirillov-Kostant-Souriau bracket on the dual of a Lie algebra
$\lie{g}$, remained the only example which allowed a formal
deformation quantization for a long time: the existence of a star
product on $\lie{g}^*$ is contained in the construction of Gutt
\cite{gutt:1983a}. The basic idea of the construction is rather
simple. The algebra of polynomial functions on $\lie{g}^*$ is
isomorphic to the symmetric algebra $\Sym^\bullet(\lie{g})$ which is,
as filtered vector space, isomorphic to the universal enveloping
algebra $\algebra{U}(\lie{g})$ via the Poincaré-Birkhoff-Witt
theorem. Taking the canonical PBW isomorphism, i.e. the total
symmetrization map, one can pull-back the product of
$\algebra{U}(\lie{g})$ to obtain a product on $\Sym^\bullet(\lie{g})$
and hence on $\Pol^\bullet(\lie{g}^*)$. Taking into account the
degrees of the polynomials in the right way allows to plug in a
deformation parameter $z$ such that one ends up with a star product
$\star_z$ on $\lie{g}^*$. This star product has many remarkable
features, one of them is that it converges for trivial reasons on the
polynomial functions $\Pol^\bullet(\lie{g}^*)$ for all $z$. This is
clear by the very construction. From the quantization point of view,
$z$ is something like $\I\hbar$ with $\hbar$ being Planck's constant.

In general, formal star products are very hard to control from an
analytic point of view: not much is known about the convergence
properties of the formal power series. The reason is rather simple: in
a formal star product the number of involved derivatives of the two
factors is typically equal or higher as the power of the deformation
parameter: hence the classical Borel lemma allows to construct smooth
functions for which the star product has radius of convergence
zero. The interesting question is of course whether one can find a
reasonable subalgebra where the star product has a nontrivial radius
of convergence. Beside very few examples, related to the Weyl-Moyal
product \cite{waldmann:2014a} and the Kähler structure on the Poincaré
disk \cite{beiser.waldmann:2014a}, not much is known in this
direction. Alternatively, there are various approaches to strict
deformation quantization based on integral formulas for the deformed
product. The formal star product then arises as asymptotic expansion
of the integrals. Here we refer to e.g. \cite{rieffel:1993a} as well
as to \cite{bieliavsky:2002a, bieliavsky.gayral:2015a}.  However, the
usage of integral formulas bounds these approaches to finite
dimensions whereas a direct investigation of the convergence of series
may still be applicable in infinite dimensions as needed for (quantum)
field-theoretic models.

In this paper we want to add yet another example where the convergence
of a formal star product can be controlled in an efficient way: the
Gutt star product on $\lie{g}^*$.

In fact, our construction will work even in some infinite-dimensional
cases. Therefore we need a slight reformulation to incorporate these
situations, too: instead of the polynomial functions we focus on the
symmetric algebra $\Sym^\bullet(\lie{g})$ of the Lie algebra. In
finite dimensions this will make no difference but in infinite
dimensions, $\Sym^\bullet(\lie{g})$ only injects into
$\Pol^\bullet(\lie{g}^*)$ but is strictly smaller. Since the Gutt star
product on $\Sym^\bullet(\lie{g})$ can be constructed in any
dimension, this seems to be a reasonable framework. Our basic idea is
now to establish a locally convex topology on $\Sym^\bullet(\lie{g})$
in such a way that the Gutt star product becomes continuous. Then it
automatically extends to the completion which we want to be as large
as possible. However, we want the completion to be small enough so
that its elements are still functions on the (topological) dual
$\lie{g}'$ of $\lie{g}$. This requires the evaluation functionals on
points in $\lie{g}'$ to be continuous. While in finite dimensions
this will work for all Lie algebras equally well, in infinite
dimensions we have to add some technical continuity properties on the
Lie bracket of $\lie{g}$.

In the following, we consider a locally convex Lie algebra $\lie{g}$
over $\mathbb{K} = \mathbb{R}$ or $\mathbb{C}$, i.e. a real or complex
locally convex topological vector space with a continuous Lie
bracket. We focus on a honestly continuous Lie bracket instead of a
separately continuous one throughout this work. This means that for
every continuous seminorm $p$ there exists another continuous seminorm
$q$ such that for all $\xi, \eta \in \lie{g}$ one has
\begin{equation}
    \label{eq:ContinuousLieBracket}
    p([\xi, \eta])
    \leq
    q(\xi) q(\eta).
\end{equation}
For our study of the Gutt star product, this will not be enough, since
we will have to control an arbitrarily high number of nested brackets
without getting a new seminorm for each bracket. Thus, we will need an
estimate which does not depend on the number of Lie brackets implied.
This motivates the following definition, see also
\cite{boseck.czichowski.rudolph:1981a}:
\begin{definition}[Asymptotic estimate algebra]
    \label{def:AE}
    Let $\algebra{A}$ be a Hausdorff locally convex algebra (not
    necessarily associative) with $\cdot$ denoting the multiplication,
    and let $p$ be a continuous seminorm.
    \begin{definitionlist}
    \item \label{item:AsymptoticEstimate} A continuous seminorm $q$ is
        said to be an asymptotic estimate for $p$, if
        \begin{equation}
            \label{Intro:AE}
            p\left(
                w_n(x_1, \ldots, x_n)
            \right)
            \leq
            q(x_1) \cdots q(x_n)
        \end{equation}
        for all words $w_n(x_1, \ldots, x_n)$ made out of $n-1$
        products of the elements $x_1, \ldots, x_n \in \algebra{A}$
        with arbitrary position of placing brackets.
    \item \label{item:AEAlgebra} A locally convex algebra is said to be
        an asymptotic estimate algebra (AE-algebra), if every continuous
        seminorm has an asymptotic estimate.
    \end{definitionlist}
\end{definition}
We are mainly interested in the case of an AE-Lie algebra but of
course also associative AE-algebras are of interest. Given an
associative AE-algebra its commutator Lie algebra is a AE-Lie
algebra. Also, all finite-dimensional Lie algebras are AE-Lie
algebras. More generally, all \emph{locally multiplicatively convex}
algebras are of this type, i.e. those where one finds a defining set
of continuous seminorms $p$ such that
\begin{equation}
    \label{eq:lmc}
    p(x \cdot y) \le p(x)p(y)
\end{equation}
for all algebra elements $x, y$. Clearly, all finite-dimensional Lie
algebras are locally multiplicatively convex.

In order to formulate our continuity results we have to specify the
topology on $\Sym_R^\bullet(\lie{g})$. Here we rely on the results of
\cite{waldmann:2014a}, where a definition for a locally convex
topology on the tensor algebra $\Tensor^\bullet(V)$ of a locally
convex vector space $V$ was given: given $R \in \mathbb{R}$, one can
define the locally convex vector spaces $\Tensor_R^\bullet(V)$ and
$\Sym_R^\bullet(V)$, where $\Sym^\bullet(V)$ denotes the symmetric
tensor algebra viewed as subspace of the tensor algebra. The basic
idea is to control the growth of tensor powers of a seminorm $p^n$
applied to the homogeneous tensor parts of degree $n$ by the $R$-th
power of $n!$. There were also definitions given for the projective
limits $\Tensor_{R^-}^\bullet(V)$ and $\Sym_{R^-}^\bullet(V)$. We will
define those spaces more precisely in Section~\ref{sec:Prelim}. Using
this topology on the symmetric algebra, we can now state the main
results:
\begin{maintheorem}
    \label{theorem:MainTheoremI}%
    Let $\lie{g}$ be an AE-Lie algebra and let $R \ge 1$.
    \begin{theoremlist}
    \item The Gutt star product $\star_z$ is continuous with respect
        to the $\Sym_R$-topology for every $z \in \mathbb{K}$.
    \item The completion $\widehat{\Sym}_R^\bullet(\lie{g})$ becomes a
        locally convex Hopf algebra with respect to the Gutt star
        product and the undeformed coproduct, antipode, and counit.
    \item The Gutt star product is convergent as series in $z \in
        \mathbb{K}$.
    \item The construction is functorial for continuous Lie algebra
        homomorphisms.
    \end{theoremlist}
\end{maintheorem}
\begin{maintheorem}
    \label{theorem:MainTheoremII}%
    Let $\lie{g}$ be a nilpotent locally convex Lie algebra. Then the
    statement of Main Theorem~\ref{theorem:MainTheoremI} holds for all
    $R \ge 1$ and for the projective limit $R \rightarrow 1^-$.
\end{maintheorem}

Remarkably, in the general case of Main
Theorem~\ref{theorem:MainTheoremI}, the completion will \emph{not}
contain the exponentials of elements in $\lie{g}$: this is in some
sense to be expected as otherwise the product of two such exponentials
would again be defined as an element of the completion. This way, one
would be able to reconstruct a convergent BCH series for all elements
in $\lie{g}$, which is known to be impossible. However, in
the nilpotent case, we
have the exponentials inside the completion. Of course, this matches
well with the fact that the BCH series is non-problematic in this
case.
\begin{remark}
    \label{remark:FiniteDimCases}%
    In the finite-dimensional case one can directly consider the
    universal enveloping algebra and establish topologies on it. The
    analogous topology to our $\Sym_R$-topology has been used and
    investigated by \cite{rasevskii:1966a} and
    \cite{goodman:1971a}. From that point of view, our results can be
    seen as a generalization to the possibly infinite-dimensional case
    including the much more involved proofs for the AE Lie algebra
    case. The strategy in \cite{goodman:1971a} is based very much on
    the fact that in finite-dimensions one has a Banach-Lie
    algebra. Moreover, the topology is obtained by a quotient
    procedure starting with the $\Tensor_R$-topology on the tensor
    algebra and proving that the ideal generated by the Lie relations
    is actually closed in order to give a Hausdorff topology on the
    universal enveloping algebra. What is missing is of course the
    continuity of the coefficient maps $C_n$ of the Gutt star product
    since there is only the filtration of the universal enveloping
    algebra but not the grading of the symmetric algebra available.
    This makes the question on analytic dependence on a deformation
    parameter meaningless: at the time of \cite{rasevskii:1966a} and
    \cite{goodman:1971a} the notion of star products was not yet
    commonly known.
\end{remark}

The paper is organized as follows: In Section~\ref{sec:Prelim} we
first outline the construction of the locally convex topologies on the
tensor algebra and the symmetric algebra according to
\cite{waldmann:2014a}. Then we recall the basic construction of the
Gutt star product and provide several equivalent descriptions. The
most useful for our purposes is the one based on the
Baker-Campbell-Hausdorff series
\cite{drinfeld:1983a}. Section~\ref{sec:LCTopolgy} contains the heart
of this work: we establish the continuity of the Gutt star product
with respect to the $\Sym_R$-topology for $R \ge 1$. The first
approach works for general asymptotic estimate Lie algebras. We also
include a second and easier proof which, however, works only in the
locally multiplicatively convex case. Then we show that we in fact
obtain an entire deformation enjoying good functorial properties with
respect to continuous Lie algebra
homomorphisms. Section~\ref{sec:Nilpotent} is devoted to the nilpotent
case.  Here we can improve the previous continuity statements to the
projective limit $R \longrightarrow 1^-$ since the BCH series has
significantly less terms in this situation. The completion will now
include exponential functions. In Section~\ref{sec:Hopf} we show the
continuity of the remaining Hopf algebra maps. This is now much
simpler as they are the classical maps not depending on the
deformation parameter $z$. Finally, Section~\ref{sec:outlook} contains
some open questions and an outlook on further research we want to
pursue in the future. In Appendix~\ref{sec:TheProofPBWGDBCH} we have
included algebraic proofs of the equivalence of various forms of the
Gutt star product, statements which are folklore knowledge but hard to
trace down in the literature.

%
% and many thanks to...
%

\medskip

\noindent
\textbf{Acknowledgements:} We would like to thank Matthias Schötz for
various discussions and suggestions. Moreover, we would like to thank
Martin Bordemann, Simone Gutt, Friedrich Wagemann, and the referee for
valuable remarks and suggestions.  Finally, we thank Jochen Wengenroth
for pointing out \cite{mitiagin.rolewicz.zelazko:1962a} via
MathOverflow.

%
% Preliminaries
%

\section{Preliminaries}
\label{sec:Prelim}

In this section we collect some preliminary results on the locally
convex topologies we use as well as on the Gutt star product.

%
% The Topologies on $\Tensor^\bullet(V)$ and $\Sym^\bullet(V)$
%

\subsection{The Topologies on $\Tensor^\bullet(V)$ and $\Sym^\bullet(V)$}
\label{subsec:TheTopologies}

Let $V$ be a locally convex vector space over $\mathbb{K}$ where
$\mathbb{K}$ stands for either $\mathbb{R}$ or $\mathbb{C}$. We want
to recall the definition of the locally convex topology on the tensor
algebra from \cite{waldmann:2014a} and some of its consequences. We
endow every tensor power $V^{\tensor n}$ with the $\pi$-topology: we
will denote for a given (continuous) seminorm $p$ its tensor power by
$p^n = p^{\tensor n}$ for $n \ge 1$. For $n = 0$ we take $p^0$ to be
the absolute value on the field $\mathbb{K}$. Then the $\pi$-topology
on $V^{\tensor_{\pi} n}$ is obtained by taking all $p^n$ for all
continuous seminorms $p$.

In order to define the Gutt star product, we need the symmetric
algebra over the underlying vector space. Recall that the
symmetrization map
\begin{equation}
    \label{eq:Symmetrizer}
    \Symmetrizer_n
    \colon
    V^{\tensor_{\pi} n}
    \longrightarrow
    V^{\tensor_{\pi} n}
    , \quad
    (v_1 \tensor \ldots \tensor v_n)
    \longmapsto
    \frac{1}{n!}
    \sum\limits_{\sigma \in S_n}
    v_{\sigma(1)}
    \tensor \ldots \tensor
    v_{\sigma(n)}
\end{equation}
is continuous and we have for all $v \in V^{\tensor_{\pi} n}$ the
estimate
\begin{equation}
    \label{eq:ContinuityForSn}
    p^n(\Symmetrizer_n(v))
    \leq
    p^n(v).
\end{equation}
Since $\Symmetrizer_n$ is idempotent, it turns out that each symmetric
tensor power
\begin{equation}
    \label{eq:SymmetricTensorPower}
    \Sym_{\pi}^n(V)
    = \image \Symmetrizer_n
    = \ker(\id - \Symmetrizer_n)
    \subseteq V^{\tensor_{\pi} n}
\end{equation}
is a closed subspace with respect to the $\pi$-topology. The symmetric
tensor product is then given by
\begin{equation}
    \label{eq:SymmetricTensorProduct}
    vw = \Symmetrizer_{n+m}(v \tensor w)
\end{equation}
for $v \in \Sym_{\pi}^n(V)$ and $w \in \Sym_{\pi}^m(V)$. Since the
tensor product obeys $p^{n+m}(v \tensor w) \le p^n(v)p^m(w)$ we get
the continuity
\begin{equation}
    \label{eq:ContinuityOfSymmetricProduct}
    p^{n + m}(vw)
    =
    p^{n+m}(\Symmetrizer_{n+m}(v \tensor w))
    \leq
    p^n(v) p^m(w)
\end{equation}
of the symmetric tensor product as well. Then the symmetric algebra
$\Sym^\bullet(V) = \bigoplus_{n=0}^\infty \Sym^n(V)$ becomes a
commutative associative unital algebra.

We now want to set up a topology on $\Tensor^\bullet(V)$ and
$\Sym^\bullet(V)$, which yields the $\pi$-topology on each component
such that the (symmetric) tensor product becomes continuous.  We
recall the following definition \cite[Def.~3.5 and
Def.~3.12]{waldmann:2014a}:
\begin{definition}[$\Tensor_R$-, $\Sym_R$-, and $\Sym_{R^-}$-topology]
    \label{definition:Topologies}%
    Let $R \in \mathbb{R}$.
    \begin{definitionlist}
    \item For every continuous seminorm $p$ on $V$ we define
        \begin{equation}
            \label{eq:pRSeminorm}
            p_R
            =
            \sum\limits_{n=0}^{\infty}
            n!^R p^n
        \end{equation}
        on the tensor algebra $\Tensor^\bullet(V)$.
    \item The locally convex topology arising from all such seminorms
        $p_R$ is called the $\Tensor_R$-topology on
        $\Tensor^\bullet(V)$, which we denote by
        $\Tensor_R^\bullet(V)$ when equipped with this topology.
    \item The induced topology on the subspace $\Sym^\bullet(V)
        \subseteq \Tensor^\bullet(V)$ is called the $\Sym_R$-topology,
        and we write $\Sym_R^\bullet(V)$.
    \item The $\Sym_{R^-}$-topology is defined as the projective limit
        of the $\Sym_{R-\epsilon}$-topologies for $\epsilon
        \longrightarrow 0$ and we set
        \begin{equation}
            \label{eq:ProjectiveLimit}
            \Sym_{R^-}^\bullet(V)
            =
            \projlim\limits_{\epsilon \longrightarrow 0}
            \Sym_{R - \epsilon}^\bullet(V).
        \end{equation}
    \end{definitionlist}
\end{definition}

We now want to collect the most important results on the locally
convex algebras $\Tensor_R^\bullet(V)$ and $\Sym_R^\bullet(V)$ which
we will later use. Proofs and more detailed explanations can be found
in \cite[Sect.~3 and Sect.~4]{waldmann:2014a}.
\begin{proposition}
    \label{proposition:Projections}%
    Let $R' \ge R \ge 0$ and let $q, p$ be continuous seminorms.
    \begin{propositionlist}
    \item \label{item:EstimateForSeminorms} We have $p_{R'} \ge p_R$
        and if $q \ge p$ then $q_R \ge p_R$.
    \item \label{item:TensorProductContinuous} The (symmetric) tensor
        product is continuous and satisfies
        \begin{equation}
            \label{eq:TensorContinuous}
            p_R(vw)
            \le
            p_R(v \tensor w)
            \le
            (2^R p)_R(v)
            (2^R p)_R(w).
        \end{equation}
    \item \label{item:PitopologyOnComponents} For all $n \in
        \mathbb{N}$ the induced topology on $\Tensor^n(V) \subseteq
        \Tensor_R^\bullet(V)$ and on $\Sym^n(V) \subseteq
        \Sym_R^\bullet(V)$ is the $\pi$-topology.
    \item \label{item:ComponentProjectionsContinuous} For all $n \in
        \mathbb{N}$ the projection and the inclusion maps $\pi_n$ and $\iota_n$
        \begin{equation}
            \label{eq:ProjectionInclusionContinuous}
            \Tensor_R^\bullet(V)
            \stackrel{\pi_n}{\longrightarrow}
            V^{\tensor_{\pi} n}
            \stackrel{\iota_n}{\longrightarrow}
            \Tensor^\bullet(V)
            \quad
            \textrm{and}
            \quad
            \Sym_R^\bullet(V)
            \stackrel{\pi_n}{\longrightarrow}
            \Sym_{\pi}^n(V)
            \stackrel{\iota_n}{\longrightarrow}
            \Sym_R^\bullet(V)
        \end{equation}
        are continuous.
    \item \label{item:CompletionExplicitly} The completions
        $\widehat{\Tensor}_R^\bullet(V)$ of $\Tensor_R^\bullet(V)$ and
        $\widehat{\Sym}_R^\bullet(V)$ of $\Sym_R^\bullet(V)$ can be
        described explicitly as
        \begin{equation}
            \label{eq:CompletionTR}
            \widehat{\Tensor}_R^\bullet(V)
            =
            \left\{
                \left.
                    v
                    =
                    \sum\limits_{n=0}^{\infty}
                    v_n
                    \ \right| \
                p_R(v)
                <
                \infty,
                \textrm{ for all continuous }
                p
            \right\}
            \subseteq
            \prod\limits_{n=0}^{\infty}
            V^{\hat{\tensor}_{\pi} n}
        \end{equation}
        and
        \begin{equation}
            \label{eq:CompletionSR}
            \widehat{\Sym}_R^\bullet(V)
            =
            \left\{
                \left.
                    v
                    =
                    \sum\limits_{n=0}^{\infty}
                    v_n
                    \ \right| \
                p_R(v)
                <
                \infty,
                \textrm{ for all continuous }
                p
            \right\}
            \subseteq
            \prod\limits_{n=0}^{\infty}
            \widehat{\Sym^n_\pi}(V),
        \end{equation}
        where the $p_R$ are extended to the Cartesian product allowing
        the value $+\infty$.
    \item \label{item:StrictlyFinerForBiggerR} If $R' > R$, then the
        topology on $\Tensor_{R'}^\bullet(V)$ is strictly finer than
        the one on $\Tensor_R^\bullet(V)$, the same holds for
        $\Sym_{R'}^\bullet(V)$ and $\Sym_R^\bullet(V)$. Therefore the
        completions for $R'$ are smaller than the ones for $R$.
    \item \label{item:ComponentInclusionsContinuous} The inclusion
        maps $\widehat{\Tensor}_{R'}^\bullet(V) \longrightarrow
        \widehat{\Tensor}_R^\bullet(V)$ and
        $\widehat{\Sym}_{R'}^\bullet(\lie{g}) \longrightarrow
        \widehat{\Sym}_R^\bullet(\lie{g})$ are continuous.
    \item \label{item:LmcJustForZero} The $\Tensor_R$-topology on
        $\Tensor_R^\bullet(V)$ and the $\Sym_R$-topology on
        $\Sym_R^\bullet(V)$ are locally multiplicatively convex with
        respect to the (symmetric) tensor product iff $R = 0$.
    \item \label{item:NuclearStuff} The space $V$ is nuclear iff
        $\Sym^\bullet_R(V)$ is nuclear iff $\Tensor^\bullet_R(V)$ is
        nuclear. In particular, this is the case if $\dim V < \infty$.
    \item \label{item:PointsArePoints} The evaluation functionals
        $\delta_\varphi\colon \Sym^\bullet_R(V) \longrightarrow
        \mathbb{K}$ for $\varphi \in V'$ are continuous.
    \end{propositionlist}
\end{proposition}

%
% The Gutt Star Product
%

\subsection{The Gutt Star Product}
\label{subsec:GuttStarProduct}

Let us now briefly recall the basic construction of the Gutt star
product according to \cite{gutt:1983a}: originally, this was just an
intermediate step to get a star product on the cotangent bundle of a
Lie group $G$ with Lie algebra $\lie{g}$. However, the resulting star
product on $\lie{g}^*$ can be described entirely algebraic as follows:
first we replace the polynomials on $\lie{g}^*$ by the symmetric
algebra $\Sym^\bullet(\lie{g})$ which behaves better in infinite
dimensions. Then we use the explicit PBW isomorphism
$\mathfrak{q}\colon \Sym^\bullet (\lie{g}) \longrightarrow
\algebra{U}(\lie{g})$ via
\begin{equation}
    \label{eq:QuantizationMap}
    \mathfrak{q}_n
    \colon
    \Sym^n(\lie{g})
    \longrightarrow
    \algebra{U}(\lie{g})
    , \quad
    \xi_1 \cdots \xi_n
    \longmapsto
    \frac{1}{n!}
    \sum\limits_{\sigma \in S_n}
    \xi_{\sigma(1)}
    \odot \cdots \odot
    \xi_{\sigma(n)}
    , \quad
    \mathfrak{q}
    =
    \sum\limits_{n = 0}^{\infty}
    \mathfrak{q}_n.
\end{equation}
We always denote the multiplication in $\algebra{U}(\lie{g})$ by
$\odot$ to avoid confusion. For $z \in \mathbb{K}$ the Gutt star
product $\star_z$ is then given by the pull-back of the product of
$\algebra{U}(\lie{g})$ together with a degree-dependent rescaling by
$z$. In detail, one defines
\begin{equation}
    \label{eq:xikStaretaell}
    x \star_z y
    =
    \sum_{n=0}^{k+\ell-1} z^n
    \pi_{k+\ell-n}
    \left(
        \mathfrak{q}^{-1}(\mathfrak{q}(x) \odot \mathfrak{q}(y))
    \right)
\end{equation}
for homogeneous $x \in \Sym^k(\lie{g})$ and $y \in \Sym^\ell(\lie{g})$
and extends this bilinearly to $\Sym^\bullet(\lie{g})$, where
$\pi_r$ projects on the homogeneous part of degree
$r$. Another approach to this star product is due to Drinfel'd who
based the construction on the Baker-Campbell-Hausdorff series
\cite{drinfeld:1983a}:
\begin{proposition}
    \label{proposition:PBWGuttDrinfeldBCHWhoElseStarProduct}%
    Let $\lie{g}$ be a Lie algebra. Then we have
    \begin{equation}
        \label{eq:DrinfeldIsGuttStar}
        \exp(\xi) \star_z \exp(\eta)
        =
        \exp \left(
            \frac{1}{z}
            \bch{z \xi}{z \eta}
	\right),
    \end{equation}
    where we consider the exponentials as formal power series in the
    variables $\xi$ and $\eta$. Conversely, by differentiating the
    right hand side with respect to these variables, one can determine
    $\star_z$ completely.
\end{proposition}
This result seems to be well-known folklore. One can find proofs based
on differential geometric arguments e.g. in
\cite[Lemma~10]{bordemann.neumaier.waldmann:1998a}. Since those
arguments do not work in infinite dimensions any more, for convenience
we give an entirely combinatorial proof in
Appendix~\ref{sec:TheProofPBWGDBCH}.  There is yet another way to
define the Gutt star product, since one can take the universal
enveloping algebra $\algebra{U}(\lie{g}_z)$ of $\lie{g}$ with the Lie
bracket rescaled by $z$: there, we have the relation
\begin{equation}
    \label{eq:bulletProductGivesBracket}
    \xi \odot \eta - \eta \odot \xi - z[\xi, \eta]
    =
    0.
\end{equation}
We use again the Poincar\'e-Birkhoff-Witt isomorphism $\mathfrak{q}_z \colon
\Sym^{\bullet}(\lie{g}) \longrightarrow \algebra{U}(\lie{g}_z)$ from Equation
\eqref{eq:QuantizationMap} and get the next, well-known result:
\begin{proposition}
    \label{proposition:DeformedLieAlgebraStarProduct}%
    Let $\lie{g}$ be a Lie algebra. Then we have for $x, y \in
    \Sym^{\bullet}(\lie{g})$ and all $z \in \mathbb{K}$
    \begin{equation}
        \label{eq:DeformedLieIsGuttStar}
        x \star_z y
        =
        \mathfrak{q}_z^{-1}
        \left(
            \mathfrak{q}_z(x)
            \odot_z
            \mathfrak{q}_z(y)
        \right).
    \end{equation}
\end{proposition}
Since we will use the isomorphism $\mathfrak{q}_z$, we also give a proof of
this in Appendix~\ref{sec:TheProofPBWGDBCH}.
\begin{remark}[Integral formula]
	Note that there is also an integral formula for the Gutt star product, which
	of course only holds in finite dimensions. This approach can be found in
	Berezin's work \cite[Formula (24)]{berezin:1967a}, for example. One can
	understand $\algebra{U}(\lie{g})$ as the distributions on the Lie Group $G$
	of $\lie{g}$ with compact support near the unit element using the convolution
	as multiplication. One uses the exponential map to get a star product on
	$\lie{g}^*$ from this. Since this formula uses the Fourier transform and the
	Baker-Campbell-Hausdorff series, one has to make sure that the Fourier
	transformed functions only have support in an area, where the exponential map
	is diffeomorphic and where the BCH series converges. For our purpose, the
	way via the formal series is more suitable since on one hand we can
	avoid those difficulties and on the other hand, our approach remains valid in
	infinite dimensions.
\end{remark}

Most of our analysis of $\star_z$ is based on properties of the BCH
series. To this end, we briefly recall the relevant facts and
establish some notation. First of all, we set for $\xi, \eta \in
\lie{g}$
\begin{equation}
    \label{Prelim:BCH}
    \bch{\xi}{\eta}
    =
    \sum\limits_{n = 1}^{\infty}
    \bchpart{n}{\xi}{\eta}
    =
    \sum\limits_{a, b = 0}^{\infty}
    \bchparts{a}{b}{\xi}{\eta},
\end{equation}
where $\bchpart{n}{\xi}{\eta}$ gives all the BCH terms which have
exactly $n$ letters and therefore $n-1$ brackets. Moreover,
$\bchparts{a}{b}{\xi}{\eta}$ stands for all BCH terms which contain
exactly $a$ times the letter $\xi$ and $b$ times the letter $\eta$. We
hence have
\begin{equation}
    \label{subsec:BCHnSumBCHab}
    \bchpart{n}{\xi}{\eta}
    =
    \sum\limits_{a + b = n}
    \bchparts{a}{b}{\xi}{\eta}
\end{equation}
with $a, b \ge 1$ for $n > 1$ and $a, b \ge 0$ for $n = 1$. The
difficulty with the BCH series is that there is no unique way to write
$\bchparts{a}{b}{\xi}{\eta}$ since one can re-arrange terms by
antisymmetry and Jacobi identity without changing the number of
$\xi$'s and $\eta$'s. Luckily, we will need only estimates for
$\bchparts{a}{b}{\xi}{\eta}$ later on. Before going into these
details, we just mention the following well-known formula (e.g. see \cite[part 2.8.12 (c)]{dixmier:1977a}) for the
lowest order terms of the BCH series:
\begin{lemma}
    \label{lemma:BCHFirstOrder}%
    Let $\lie{g}$ be a Lie algebra and $\xi, \eta \in \lie{g}$. Then
    we can write the Baker-Campbell-Hausdorff series up to first order
    in $\eta$ as
    \begin{equation}
        \label{Prelim:BCHFirstOrder}
        \bch{\xi}{\eta}
        =
        \xi
        +
        \sum\limits_{n = 0}^{\infty}
        \frac{B_n^*}{n!}
        \left( \ad_{\xi} \right)^n (\eta)
        +
        \mathcal{O}(\eta^2),
    \end{equation}
    where the Bernoulli numbers $B_n^*$ are defined by the series
    \begin{equation}
        \label{Prelim:BernoulliNumbers}
        \frac{z}{1 - e^{-z}}
        =
        \sum\limits_{n = 0}^{\infty}
        \frac{B_n^*}{n!}
        z^n.
    \end{equation}
\end{lemma}
Recall that the series \eqref{Prelim:BernoulliNumbers} converges
absolutely for $|z| < 2\pi$ and that we have the (quite rough)
estimate
\begin{equation}
    \label{eq:BnEstimate}
    |B_n^*| \le n!
\end{equation}
for the Bernoulli numbers.

%
% A Formula for the Gutt Star Product
%

\subsection{A Formula for the Gutt Star Product}
\label{subsec:FormulaGuttStarProduct}

In a next step, we develop some technical tools which will be useful
for proving the continuity of the Gutt star product. First we write
the Gutt star product for $x, y \in \Sym^\bullet(\lie{g})$ as
\begin{equation}
    \label{eq:xGuttStary}
    x \star_z y
    =
    \sum_{n=0}^\infty z^n C_n(x, y)
\end{equation}
with bilinear operators $C_n\colon \Sym^\bullet(\lie{g}) \times
\Sym^\bullet(\lie{g}) \longrightarrow \Sym^\bullet(\lie{g})$ as
usual. Now, the main goal is to use
Proposition~\ref{proposition:PBWGuttDrinfeldBCHWhoElseStarProduct} and
the BCH series to obtain fairly explicit formulas for the
contributions $C_n(x, y)$, up to the knowledge of the BCH series.

From the original work of Gutt \cite[Prop.~1]{gutt:1983a}, see also
\cite[2.8.12 (c)]{dixmier:1977a} as well as \cite[Rem.~5.2.8]{neumaier:1998a}
and \cite[Eq.~2.23]{kathotia:1998a:pre}, we get the following formula for
$\star_z$ whenever one factor, say the second, is linear:
\begin{proposition}
    \label{proposition:GuttStarOneLinearFactor}%
    Let $z \in \mathbb{K}$.
    \begin{propositionlist}
    \item For all $\xi, \eta \in \lie{g}$ and $k \in \mathbb{N}$ we
        have
        \begin{equation}
            \label{Formulas:LinearMonomial1}
            \xi^k \star_z \eta
            =
            \sum\limits_{j=0}^k
            \binom{k}{j} z^j B_j^*
            \xi^{k-j}
            \left( \ad_{\xi} \right)^j (\eta).
        \end{equation}
    \item For  all $k \in \mathbb{N}$ and $\xi_1, \ldots, \xi_k, \eta
        \in \lie{g}$ we have
        \begin{equation}
            \label{Formulas:LinearMonomial2}
            \xi_1 \cdots \xi_k \star_z \eta
            =
            \sum\limits_{j=0}^k
            \frac{1}{k!} \binom{k}{j}
            z^j B_j^*
            \sum\limits_{\sigma \in S_k}
            \xi_{\sigma(1)} \cdots \xi_{\sigma(k - j)}
            [\xi_{\sigma(k - j + 1)},
            [ \ldots [\xi_{\sigma(k)}, \eta] \ldots ]
            ]
            .
        \end{equation}
    \end{propositionlist}
\end{proposition}
\begin{proof}
    For convenience, we sketch the proof: the first part is
    essentially Lemma~\ref{lemma:BCHFirstOrder} together with
    Proposition~\ref{proposition:PBWGuttDrinfeldBCHWhoElseStarProduct}.
    The second is then obtained by polarization from the first:
    Set
    \begin{equation*}
        \Xi
        =
        \Xi(t_1, \ldots, t_k)
        =
        \sum\limits_{j=1}^k
        t_j \xi_j
    \end{equation*}
    for parameters $t_1, \ldots, t_k \in \mathbb{R}$. Then we get by
    differentiating
    \begin{equation*}
        \frac{\partial^k}{\partial t_1 \cdots \partial t_k}
        \Xi^k
        =
        k! \xi_1 \cdots \xi_k.
    \end{equation*}
    This formal differentiation in Equation~\eqref{Formulas:LinearMonomial1}
    gives the result.
\end{proof}

As a consequence, the previous proposition determines the explicit form of the
bilinear operators $C_n$ of the Gutt star product whenever one factor
is linear. In particular, by associativity we have
\begin{equation}
    \label{Eq:Formulas:MultipleStars}
    \xi_1 \star_z \cdots \star_z \xi_k
    =
    \sum\limits_{\substack{
        1 \leq j \leq k-1 \\
        i_j \in \{0, \ldots, j\}
      }
    }
    z^{i_1 + \cdots + i_{k-1}}
    C_{i_{k-1}}
    \left(
        \cdots
        C_{i_2}\left(
            C_{i_1}\left(\xi_1, \xi_2 \right), \xi_3
        \right),
        \ldots, \xi_{k}
    \right)
\end{equation}
for $2\leq k \in \mathbb{N}$ and $\xi_1, \ldots, \xi_k \in \lie{g}$.
While this formula is easy to handle for small $k$, it becomes quite
tedious in general.

For our continuity estimates we need to go beyond linear terms: both
factors in the product have to be general. By differentiating the identity
form Proposition~\ref{proposition:PBWGuttDrinfeldBCHWhoElseStarProduct},
we get the next result.
\begin{lemma}
    \label{lemma:2MonomialsFormula1}%
    Let $\lie{g}$ be a Lie algebra.
    \begin{lemmalist}
    \item \label{item:Gstar2Monomials}
        Let $\xi, \eta \in \lie{g}$. We have
        \begin{equation}
            \xi^k \star_z \eta^\ell
            =
            \sum\limits_{n=0}^{k + \ell - 1}
            z^n
            C_n \left(\xi^k, \eta^\ell \right),
        \end{equation}
        with $C_0(\xi^k, \eta^\ell) = \xi^k\eta^\ell$ and
        \begin{equation}
            \label{Formulas:2MonomialsFormula1}
            C_n \left(\xi^k, \eta^{\ell} \right)
            =
            \frac{k! \ell!}{(k + \ell - n)!}
            \sum\limits_{\substack{a_1, b_1, \ldots, a_r, b_r \geq 0 \\
                a_i + b_i \geq 1 \\
                a_1 + \cdots + a_r = k \\
                b_1 + \cdots + b_r = \ell
              }
            }
            \bchparts{a_1}{b_1}{\xi}{\eta}
            \cdots
            \bchparts{a_r}{b_r}{\xi}{\eta}
        \end{equation}
        for $k, \ell \in \mathbb{N}$ and $n \ge 1$, where we set $r =
        k + \ell - n$ for abbreviation.
    \item \label{item:Gstar2MonomialsPolarized} Denote by
        $\bchtilde{a}{b}{\argument}{\argument}$ the unique
        $a+b$-linear map, symmetric in the first $a$ and in the last
        $b$ arguments, such that
        \begin{equation}
            \label{Formulas:BCHTilde}
            \bchtilde{a}{b}{\xi, \ldots, \xi}{\eta, \ldots, \eta}
            =
            \bchparts{a}{b}{\xi}{\eta}
        \end{equation}
        for $\xi, \eta \in \lie{g}$. Then we have
        \begin{align}
            \label{Formulas:2MonomialsFormula2}
            C_n \left(
                \xi_1 \cdots \xi_k; \eta_1 \cdots \eta_{\ell}
            \right)
            & =
            \frac{1}{(k + \ell - n)!}
            \sum\limits_{\sigma \in S_k, \tau \in S_{\ell}}
            \sum\limits_{\substack{a_1, b_1, \ldots, a_r, b_r \geq 0 \\
                a_i + b_i \geq 1 \\
                a_1 + \cdots + a_r = k \\
                b_1 + \cdots + b_r = \ell
              }
            }
            \\
            \nonumber
            & \qquad
            \bchtilde{a_i}{b_i}
            {\xi_{\sigma(1)}, \ldots, \xi_{\sigma(a_1)}}
            {\eta_{\tau(1)}, \ldots, \eta_{\tau(b_1)}}
            \cdots
            \\
            \nonumber
            & \qquad
            \bchtilde{a_r}{b_r}
            {\xi_{\sigma(k - a_r + 1)}, \ldots, \xi_{\sigma(k)}}
            {\eta_{\tau(\ell - b_r + 1)}, \ldots, \eta_{\tau(\ell)}}.
        \end{align}
        for $\xi_1, \ldots, \xi_k, \eta_1, \ldots, \eta_{\ell} \in
        \lie{g}$.
    \end{lemmalist}
\end{lemma}
\begin{proof}
    We consider $z \ne 0$ since the claim is trivial for $z =
    0$. Using
    Proposition~\ref{proposition:PBWGuttDrinfeldBCHWhoElseStarProduct}
    we get the star product of $\xi^k$ and $\eta^\ell$ by
    differentiating
    \begin{align*}
        \xi^k \star_z \eta^\ell
        &=
        \frac{\partial^k}{\partial t^k}
        \frac{\partial^{\ell}}{\partial s^{\ell}}
        \At{t,s = 0}
        \exp\left(
            \frac{1}{z}
            \bch{zt\xi}{sz\eta}
        \right) \\
        &=
        \frac{\partial^k}{\partial t^k}
        \frac{\partial^{\ell}}{\partial s^{\ell}}
        \At{t,s = 0}
        \sum_{r=0}^\infty \frac{1}{r!} \frac{1}{z^r}
        \left(\bch{zt\xi}{zs\eta}\right)^r \\
        &=
        \frac{\partial^k}{\partial t^k}
        \frac{\partial^{\ell}}{\partial s^{\ell}}
        \At{t,s = 0}
        \sum_{r=0}^\infty \frac{1}{r!} \frac{1}{z^r}
        \left(
            \sum_{j=0}^{k+\ell}
            \bchpart{j}{zt\xi}{zs\eta}
        \right)^r \\
        &=
        \sum_{r=0}^{k+\ell} \frac{1}{r!} \frac{z^{k+\ell}}{z^r}
        k!\ell!
        \sum_{\substack{
            a_1, b_1, \ldots, a_r, b_r \geq 0 \\
            a_i + b_i \geq 1 \\
            a_1 + \cdots + a_r = k \\
            b_1 + \cdots + b_r = \ell
          }
        }
        \bchparts{a_i}{b_i}{\xi}{\eta}
        \cdots
        \bchparts{a_r}{b_r}{\xi}{\eta}.
    \end{align*}
    Here we have used that the $k$-th derivative by $t$ at $t = 0$
    gives $k!$ times the coefficient of $\xi^k$ and analogously for
    the $\ell$-th derivative by $s$. The bounds on the parameters
    $a_1, \ldots, a_r, b_1, \ldots, b_r$ and thereby the bounds on $r$
    originate from the fact that the BCH series has no constant
    term. Moreover, higher powers of the BCH series as $k + \ell$ will
    clearly not contain terms we need. The second formula is obtained by
    polarizing the first: We introduce again parameters $t_i, s_j \in
    \mathbb{R}$ with $i = 1, \ldots, k$ and $j = 1, \ldots, \ell$ setting
    \begin{equation*}
    	    \Xi
    	    =
    	    \sum\limits_{i=1}^k
    	    t_i \xi_i
    	    \quad \text{ and } \quad
    	    \Eta
    	    =
    	    \sum\limits_{j=1}^{\ell}
    	    s_j \eta_j.
    \end{equation*}
    Differentiating $C_n(\Xi, \Eta)$ as in
    Formula~\eqref{Formulas:2MonomialsFormula1} once by every
    parameter and dividing by $k! \ell!$, we get permutations of all
    $\xi_i$ and all $\eta_j$ and hence
    Equation~\eqref{Formulas:2MonomialsFormula2}.
\end{proof}
\begin{remark}
    The importance of this formula is that we have reduced the
    complexity of $\star_z$ to the difficulties to compute the
    homogeneous parts of the BCH series. This is of course still a
    complicated and tedious problem but luckily we are only interested
    in estimating the terms $\bchparts{a}{b}{\xi}{\eta}$ instead of
    computing them explicitly.
\end{remark}

%
% The Continuity of $\star_z$
%

\section{The Continuity of $\star_z$}
\label{sec:LCTopolgy}

The next step is finding continuity estimates for $\star_z$. All
estimates which are done in the next three sections follow mostly the
same scheme: We extend maps from $\Sym_R^\bullet(\lie{g})$ to
$\Tensor_R^\bullet(\lie{g})$ by using the symmetrization map
$\Symmetrizer$ beforehand. Our first examples for this are the Gutt
star product and the $C_n$-operators from Equation
\eqref{eq:xGuttStary}. Set
\begin{equation}
    \label{eq:TheOstarMap}
    \ostar_z \colon
    \Tensor_R^\bullet(\lie{g})
    \tensor
    \Tensor_R^\bullet(\lie{g})
    \longrightarrow
    \Sym_R^\bullet(\lie{g})
    \quad
    \textrm{with}
    \quad
    \ostar_z
    =
    \star_z
    \circ
    (\Symmetrizer \tensor \Symmetrizer)
\end{equation}
and analogously for the $C_n$. It is clear that all extended maps
coincide with the original maps on $\Sym_R^\bullet(\lie{g})$. Then, we
use the AE-property for the seminorms (which is always valid for
locally convex nilpotent Lie algebras), to estimate Lie
brackets. Finally, we use a feature of the projective tensor product,
in order to generalize statements about factorizing tensors to arbitrary
ones. This is done once explicitly at the end of the proof of
Proposition~\ref{Prop:LCAna:Continuity1}. We then just refer to this
construction since it always works analogously.

%
% A Direct Continuity Result
%

\subsection{A Direct Continuity Result}
\label{subsec:DirectContinuity}

For a word $w$ in the two letters $\xi$ and $\eta$ we denote by $[w]$
the unique Lie bracket expression of this word, where we have nested
the Lie brackets to the left in the sense that
\begin{equation}
    \label{eq:LieLeftNestedWord}
    [\xi \eta \eta \ldots ]
    =
    [ \ldots [[\xi, \eta], \eta], \ldots ].
\end{equation}
Moreover, $|w|$ denotes the number of letters, i.e. the length of the
word.
\begin{lemma}
    \label{lemma:BCHTermsEstiamte}%
    Let $\lie{g}$ be a AE-Lie algebra, $p$ a continuous seminorm, $q$
    an asymptotic estimate for it.
    \begin{lemmalist}
    \item \label{item:GoldbergThompsonBCH} For $n \in \mathbb{N}$
        there are numbers $g_w \in \mathbb{Q}$ such that for $\xi,
        \eta \in \lie{g}$ one has
        \begin{equation}
            \label{LCAna:GoldbergThompsonBCH}
            \bchpart{n}{\xi}{\eta}
            =
            \sum\limits_{|w| = n}
            \frac{g_w}{n}
            [w].
        \end{equation}
    \item \label{item:ThompsonEstimate} The coefficients $g_w$ can be
        chosen to fulfil the estimate
        \begin{equation}
            \label{LCAna:ThompsonEstimate}
            \sum\limits_{|w| = n}
            \left| \frac{g_w}{n} \right|
            \leq
            \frac{2}{n}.
        \end{equation}
    \item \label{item:AEonWords} For every word $w$ which consists
        of $a$ times the letter $\xi$ and $b$ times the letter $\eta$,
        we have
        \begin{equation}
            \label{LCAna:AEonWords}
            p([w])
            \leq
            q(\xi)^a q(\eta)^b.
        \end{equation}
    \item \label{item:BCHEstimate}
        Let $a,b \in \mathbb{N}$ and $\xi_1, \ldots, \xi_a, \eta_1, \ldots,
        \eta_b \in \lie{g}$. We have the estimate
        \begin{equation}
            \label{LCAna:BCHEstimate}
            p \left(
                \bchtilde{a}{b}
                {\xi_1, \ldots, \xi_a}
                {\eta_1, \ldots, \eta_b}
            \right)
            \leq
            \frac{2}{a + b}
            q(\xi_1) \cdots q(\xi_a)
            q(\eta_1) \cdots q(\eta_b).
        \end{equation}
    \end{lemmalist}
\end{lemma}
\begin{proof}
    Goldberg found a form of writing the BCH series as a series of
    words in two letters $X$ and $Y$ with certain coefficients
    \cite{goldberg:1956a} called $c_X(s_1, \ldots, s_m)$ or $c_Y(s_1,
    \ldots, s_m)$ depending whether a word begins with the letter $X$
    or $Y$. Here, $c_X(s_1, \ldots, s_m)$ belongs to the word
    \begin{equation*}
        X^{s_1} Y^{s_2} \ldots (X \textrm{ or } Y)^{s_m},
    \end{equation*}
    where the word ends with the letter $X$ or $Y$ if $m$ is odd or
    even, respectively. In \cite{thompson:1982a}, Thompson put this
    into Lie bracket form and proved that using $g_w = c_X$ [or $g_w =
    c_Y$] one gets identity \eqref{LCAna:GoldbergThompsonBCH}. In
    \cite{thompson:1989a} Thompson put estimates on these coefficients
    and proved the estimate \eqref{LCAna:ThompsonEstimate}. The
    inequality \eqref{LCAna:AEonWords} is due to the AE-property,
    which does not see the way how brackets are set but just counts
    the number of $\xi$'s and $\eta$'s in the whole expression.  We
    use the notation $|w|_{\xi}$ for the number of $\xi$'s appearing
    in a word $w$ and $|w|_{\eta}$ for the number of
    $\eta$'s. Clearly, $|w| = |w|_{\xi} + |w|_{\eta}$. With
    \eqref{LCAna:ThompsonEstimate} and the AE-property of $\lie{g}$,
    we get
    \begin{align*}
        p \left(
            \bchtilde{a}{b}
            {\xi_1, \ldots, \xi_a}
            {\eta_1, \ldots, \eta_b}
        \right)
        & \leq
        \sum\limits_{\substack{
            |w|_{\xi} = a \\
            |w|_{\eta} = b
          }}
        p \left(
            \frac{g_w}{a + b}
            [w]
        \right)
        \\
        & \leq
        \sum\limits_{\substack{
            |w|_{\xi} = a \\
            |w|_{\eta} = b
          }}
        \left| \frac{g_w}{a + b} \right|
        q(\xi_1) \cdots q(\xi_a)
        q(\eta_1) \cdots q(\xi_b)
        \\
        & \leq
        \frac{2}{a + b}
        q(\xi_1) \cdots q(\xi_a)
        q(\eta_1) \cdots q(\eta_b).
    \end{align*}
\end{proof}

In a next step, we want to approach the estimate via the formula
\begin{equation}
    \label{eq:GstarOfXisAndEtas}
    \xi_1 \cdots \xi_k \star_z \eta_1 \cdots \eta_{\ell}
    =
    \sum\limits_{n=0}^{k + \ell -1}
    z^n
    C_n (\xi_1 \cdots \xi_k, \eta_1 \cdots \eta_{\ell}).
\end{equation}
To shorten the very long expression from
Equation~\eqref{Formulas:2MonomialsFormula2}, we occasionally
abbreviate the summations by
\begin{equation}
    C_n\left(
        \xi_1 \cdots \xi_k,
        \eta_1 \cdots \eta_{\ell}
    \right)
    =
    \frac{1}{r!}
    \sum\limits_{\sigma, \tau}
    \sum\limits_{a_i, b_j}
    \bchtilde{a_1}{b_1}{\xi_{\sigma(i)}}{\eta_{\tau(j)}}
    \cdots
    \bchtilde{a_r}{b_r}{\xi_{\sigma(i)}}{\eta_{\tau(j)}},
\end{equation}
meaning the summations as given in
Lemma~\ref{lemma:2MonomialsFormula1} and using $r = k + \ell - n$.
\begin{proposition}
    \label{Prop:LCAna:Continuity1}%
    Let $\lie{g}$ be an AE-Lie algebra, $R \geq 0$, $p$ a continuous
    seminorm with an asymptotic estimate $q$, and $z \in \mathbb{K}$.
    \begin{propositionlist}
    \item \label{item:CnOperatorEstimate} For $n \in \mathbb{N}$, the
        operator $C_n$ is continuous and for all $x, y \in
        \Tensor_R^{\bullet}(\lie{g})$ we have the estimate
    	\begin{equation}
            \label{LCAna:CnOperators}
            p_R \left( C_n(x,y) \right)
            \leq
            \frac{n!^{1 - R}}{2 \cdot 8^n}
            (32 q)_R (x)
            (32 q)_R (y).
    	\end{equation}
    \item \label{item:LCAna:Continuity1} For $R \geq 1$, the Gutt star
        product is continuous and for all $x, y \in
        \Tensor_R^{\bullet}(\lie{g})$ we have the estimate:
        \begin{equation}
            \label{LCAna:Continuity1}
            p_R(x \ostar_z y)
            \leq
            (c q)_R(x) (c q)_R(y)
        \end{equation}
        with $c = 32(|z| + 1)$. Hence, the estimate
        \eqref{LCAna:Continuity1} holds on
        $\widehat{\Sym}_R^\bullet(\lie{g})$ for all $z \in
        \mathbb{K}$, too.
    \end{propositionlist}
\end{proposition}
\begin{proof}
    Let us use $r = k + \ell - n$ as before and recall that the
    products are taken in the symmetric algebra.  Then we can use
    Equation \eqref{Formulas:2MonomialsFormula1} from
    Lemma~\ref{lemma:2MonomialsFormula1} and put estimates on it.  Let
    $p$ be a continuous seminorm and let $q$ be an asymptotic estimate
    for it. Then we get
    \begin{align*}
        p_R \big(
            C_n \big(
                \xi_1 \tensor \cdots \tensor \xi_k, &
                \eta_1 \tensor \cdots \tensor \eta_{\ell}
            \big)
        \big)
        =
        p_R \bigg(
         	\frac{1}{r!}
			\sum\limits_{\sigma, \tau}
			\sum\limits_{a_i, b_j}
			\bchtilde{a_1}{b_1}{\xi_{\sigma(i)}}{\eta_{\tau(j)}}
			\cdots
			\bchtilde{a_r}{b_r}{\xi_{\sigma(i)}}{\eta_{\tau(j)}}
        \bigg)
        \\
        & \ot{(a)}{\leq}
        \frac{1}{r!}
        r!^R
        \sum\limits_{\sigma, \tau}
		\sum\limits_{a_i, b_j}
        p \left(
            \bchtilde{a_1}{b_1}{\xi_{\sigma(i)}}{\eta_{\tau(j)}}
        \right)
        \cdots
        p \left(
            \bchtilde{a_r}{b_r}{\xi_{\sigma(i)}}{\eta_{\tau(j)}}
        \right)
        \\
        & \ot{(b)}{\leq}
        \frac{1}{r!^{1-R}}
        \sum\limits_{\sigma, \tau}
		\sum\limits_{a_i, b_j}
		\frac{2}{a_1 + b_1}
		\ldots
		\frac{2}{a_r + b_r}
		q(\xi_1) \cdots q(\xi_k)
		q(\eta_1) \cdots q(\eta_{\ell})
        \\
        & \ot{(c)}{\leq}
		q(\xi_1) \cdots q(\xi_k)
		q(\eta_1) \cdots q(\eta_{\ell})
        2^r
        \frac{k! \ell!}{r!^{1-R}}
        \sum\limits_{a_i, b_j}
		1,
    \end{align*}
    where we just used the continuity estimate for the symmetric
    tensor product in ($a$), Lemma~\ref{lemma:BCHTermsEstiamte},
    \refitem{item:BCHEstimate}, in ($b$) and $\frac{2}{a_i + b_i} \leq
    2$ in ($c$). We estimate the number of terms in the sum and get
    \begin{equation*}
        \sum\limits_{\substack{a_1, b_1, \ldots, a_r, b_r \geq 0 \\
            a_i + b_i \geq 1 \\
            a_1 + \cdots + a_r = k \\
            b_1 + \cdots + b_r = \ell
          }}
		1
		\leq
        \sum\limits_{\substack{a_1, b_1, \ldots, a_r, b_r \geq 0 \\
            a_1 + b_1 + \cdots + a_r + b_r = k + \ell
          }}
		1
		=
		\binom{k + \ell + 2r - 1}{k + \ell}
		\leq
		2^{3(k + \ell) - 2n - 1}.
    \end{equation*}
    Using this estimate, we get
    \begin{align*}
        &p_R \big(
            C_n \big(
                \xi_1 \tensor \cdots \tensor \xi_k,
                \eta_1 \tensor \cdots \tensor \eta_{\ell}
            \big)
        \big) \\
        &\quad\leq
        q(\xi_1) \cdots q(\xi_k)
		q(\eta_1) \cdots q(\eta_{\ell})
        2^{k + \ell - n}
        \frac{k! \ell!}{(k + \ell - n)!^{1-R}}
        2^{3(k + \ell) - 2n - 1}
        \\
        &\quad=
        q_R \left(
            \xi_1 \tensor \cdots \tensor \xi_k
        \right)
        q_R \left(
            \eta_1 \tensor \cdots \tensor \eta_{\ell}
        \right)
        2^{4(k + \ell) - 3n - 1}
        \left(
        	\frac{k! \ell! n!}{(k + \ell - n)! n!}
        \right)^{1-R}
        \\
        &\quad\leq
        q_R \left(
            \xi_1 \tensor \cdots \tensor \xi_k
        \right)
        q_R \left(
            \eta_1 \tensor \cdots \tensor \eta_{\ell}
        \right)
        2^{4(k + \ell) - 3n - 1}
        2^{(1 - R)(k + \ell)}
        n!^{1 - R}
        \\
        &\quad=
        \frac{n!^{1 - R}}{2 \cdot 8^n}
        (32 q)_R \left(
            \xi_1 \tensor \cdots \tensor \xi_k
        \right)
        (32 q)_R \left(
            \eta_1 \tensor \cdots \tensor \eta_{\ell}
        \right).
    \end{align*}
    The estimate \eqref{LCAna:CnOperators} is now proven on factorizing
    tensors. For general tensors $x,y \in \Tensor_R^{\bullet}(\lie{g})$,
    we use the following argument: let
    \begin{equation*}
     	x
     	=
     	\sum\limits_{m=0}^k
     	x^{(m)}
     	=
     	\sum\limits_{m=0}^k
     	\sum_i
     	x_i^{(m)}
     	, \quad
     	y
     	=
     	\sum\limits_{n=0}^{\ell}
     	y^{(n)}
     	=
     	\sum\limits_{n=0}^{\ell}
     	\sum_j
     	y_j^{(n)},
	\end{equation*}
	where the $x_i^{(m)}$ and the $y_j^{(n)}$ are factorizing
        tensors of homogeneous degrees $m$ and $n$, respectively, with
        the maximal degree of $x$ and $y$ being $k$ and $\ell$,
        respectively. Hence we have
	\begin{equation*}
		x_i^{(m)}
		=
		x_i^{(m),1}
     	\tensor \cdots \tensor
     	x_i^{(m),m}
     	, \quad
		y_j^{(n)}
		=
		y_j^{(n),1}
     	\tensor \cdots \tensor
     	y_j^{(n),n}.
    \end{equation*}
	Now we get the following estimate:
    \begin{align*}
     	p_R \left( C_n (x, y) \right)
    	    & =
    	    p_R
    	    \left(
    		    C_n
    		    \left(
		     	\Bigg(
   			 		\sum\limits_{m=0}^k
	   		 		\sum_i
	    			    x_i^{(m)}
		    	    \Bigg)
	    		    ,
	    		    \Bigg(
		    		    \sum\limits_{n=0}^{\ell}
		    		    \sum_j
	    			    y_j^{(n)}
	    		    \Bigg)
	     	\right)
     	\right)
     	\\
     	& =
     	p_R
     	\left(
   		 	\sum\limits_{m=0}^k
   		 	\sum\limits_{n=0}^{\ell}
   		 	\sum_i
    		    \sum_j
    		    C_n
    		    \left(
    			    x_i^{(m)}
    			    ,
    			    y_j^{(n)}
    		    \right)
     	\right)
     	\\
     	& \leq
   		\sum\limits_{m=0}^k
   		\sum\limits_{n=0}^{\ell}
   		\sum_i
     	\sum_j
     	\frac{n!^{1 - R}}{2 \cdot 8^n}
     	(32 q)_R \left( x_i^{(m)} \right)
     	(32 q)_R \left( y_j^{(n)} \right)
     	\\
     	& =
     	\frac{n!^{1 - R}}{2 \cdot 8^n}
   		\left(
	   		\sum\limits_{m=0}^k
   			\sum_i
   			(32 q)_R \left( x_i^{(m)} \right)
   		\right)
   		\left(
	   		\sum\limits_{n=0}^{\ell}
   			\sum_j
    		(32 q)_R \left( y_j^{(n)} \right)
     	\right)
     	\\
     	& =
     	\frac{n!^{1 - R}}{2 \cdot 8^n}
   		\left(
	   		\sum\limits_{m=0}^k
	   		32^m m!^R
   			\sum_i
   			q \left( x_i^{(m), 1} \right)
   			\cdots
   			q \left( x_i^{(m), m} \right)
   		\right)
   	    \\
   	    & \qquad
   	    \cdot
   		\left(
	   		\sum\limits_{n=0}^{\ell}
   			32^n n!^R
   			\sum_j
    		q \left( y_j^{(n), 1} \right)
    		\cdots
    		q \left( y_j^{(n), n} \right)
     	\right).
    \end{align*}
    We have to take the infimum on both sides over all
    representations of the $x^{(m)}$ and the $y^{(n)}$. On the
    right hand side, we get for the $x$-terms
    \begin{equation*}
     	\inf \left\{ \left.
    			\sum_i
    			q \left( x_i^{(m), 1} \right)
    			\cdots
    			q \left( x_i^{(m), m} \right)
    		\right|
    			\sum_i
    			x_i^{(m), 1}
    			\tensor \cdots \tensor
    			x_i^{(m), m}
    			=
    			x^{(m)}
     	\right\}
     	=
     	q^m \left( x^{(m)} \right).
    \end{equation*}
    The $y$-terms give in the same way $q^n \left( y^{(n)} \right)$.
    This is exactly the definition of the tensor power of a seminorm
    as needed for the projective tensor product.  We can recollect the
    factorials and the coefficients and get
    \begin{align*}
     	p_R \left( C_n(x, y) \right)
     	& \leq
    	    \frac{n!^{1 - R}}{2 \cdot 8^n}
   		\left(
	   		\sum\limits_{m=0}^k
	   		32^k k!^R
   			\sum_i
   			q^m \left( x_i^{(m)} \right)
   		\right)
   		\left(
	   		\sum\limits_{n=0}^{\ell}
   			32^n n!^R
   			\sum_j
    		q^n \left( y_j^{(n)} \right)
        \right)
    	    \\
        & =
        \frac{n!^{1 - R}}{2 \cdot 8^n}
        \left(
            \sum\limits_{m=0}^k
   			(32 q)_R \left( x^{(m)} \right)
        \right)
        \left(
   			\sum\limits_{n=0}^{\ell}
   			(32 q)_R \left( y^{(n)} \right)
   		\right)
    		\\
     	& =
     	\frac{n!^{1 - R}}{2 \cdot 8^n}
     	(32 q)_R(x)
     	(32 q)_R(y),
    \end{align*}
    which proves \eqref{LCAna:CnOperators} on general tensors.
    For the second statement, let $x$ and $y$ be tensors of degree
    at most $k$ and $\ell$ respectively. We have
    \begin{align}
    \nonumber
    	p_R \left(
    	x \star_z y
    	\right)
    	& =
    	p_R \left(
            \sum\limits_{n=0}^{k + \ell - 1}
            z^n C_n(x, y)
    	\right)
    	\\
    \nonumber
    	& \leq
    	\sum\limits_{n=0}^{k + \ell - 1}
    	p_R \left(
    		z^n C_n(x, y)
    	\right)
    	\\
    \label{LCAna:2MonomialEstimate}
    	& \ot{(a)}{\leq}
    	\sum\limits_{n=0}^{k + \ell - 1}
    	\frac{|z|^n}{2 \cdot 8^n}
    	n!^{1 - R}
    	(32 q)_R(x)
    	(32 q)_R(y)
    	\\
    \nonumber
    	& \ot{(b)}{\leq}
    	\frac{(|z| + 1)^{k + \ell}}{2}
    	(32 q)_R(x)
    	(32 q)_R(y)
    	\sum\limits_{n = 0}^{\infty}
    	\frac{1}{8^n}
		\\
	\nonumber
		& \leq
    	(32(|z| + 1) q)_R(x)
    	(32(|z| + 1) q)_R(y),
    \end{align}
    by using \eqref{LCAna:CnOperators} in (a) and $R \geq 1$ in (b).
    Since estimates on $\Sym_R^{\bullet}(\lie{g})$ also hold for the
    completion, the second part is done and hence the first part of
    our Main Theorem \ref{theorem:MainTheoremI} is proven.
\end{proof}

It is easy to see that we need at least $R \geq 1$ to get rid of the
factorials which come up because of the combinatorics of the star
product, but it is interesting to know that the Gutt star product
really fails continuity, if $R < 1$:
\begin{example}
    \label{Ex:LCAna:HeisenbergAlgebra}%
    Let $0 \leq R < 1$ and $\lie{g}$ be the Heisenberg algebra in three
    dimensions, i.e. the Lie algebra generated by the elements
    $P$, $Q$ and $E$ with the bracket $[P,Q] = E$ and all other brackets
    vanish. We impose on $\lie{g}$ the $\ell^1$-topology with the norm $n$ and
    $n(P) = n(Q) = n(E) = 1$. Then we consider the sequences
    \begin{equation*}
        a_k
        =
        \frac{P^k}{k!^{R + \epsilon}}
        \quad
        \textrm{and}
        \quad
        b_k
        =
        \frac{Q^k}{k!^{R + \epsilon}}
    \end{equation*}
    with $2 \epsilon < 1 - R$. It is easy to see that
    \begin{equation*}
        n_R(a_k)
        =
        n_R(b_k)
        =
        k!^{- \epsilon}
    \end{equation*}
    and hence we get the limit for any $c > 0$ by
    \begin{equation*}
     	\lim_{k \longrightarrow \infty}
     	(cn)_R(a_k)
     	=
     	\lim_{k \longrightarrow \infty}
     	(cn)_R(a_k)
     	=
     	0
    \end{equation*}
    We want to show that there is no $c > 0$ such that
    \begin{equation*}
        n_R(a_k \star_z b_k)
        \leq
        (c n)_R(a_k) (c n)_R(b_k).
    \end{equation*}
    In other words, $n_R(a_k \star_z b_k)$ grows faster than
    exponentially. But this is the case, since we can calculate the star
    product explicitly and see
    \begin{align*}
        n_R(a_k \star_z b_k)
        & =
        n_R \left(
        \sum\limits_{j=0}^k
        \binom{k}{j}
        \binom{k}{j}
        j! \frac{1}{k!^{2R + 2 \epsilon}}
        P^{k-j} Q^{k-j} E^j
        \right)
        \\
        & =
        \sum\limits_{j=0}^k
        \frac{k!^2 j! (2k - j)!^R}{(k-j)!^2 j!^2 k!^{2R + 2 \epsilon}}
        \underbrace{
        n^{2k-j}
        ( P^{k-j} Q^{k-j} E^j )
        }_{= 1}
        \\
        & =
        \sum\limits_{j=0}^k
        \underbrace{
        \binom{k}{j}^2 \binom{2k}{k} \binom{2k}{j}^{-1}
        }_{\geq 1}
        \frac{j!^{1-R}}{k!^{2 \epsilon}}
        \\
        & \geq
        \sum\limits_{j=0}^k
        \frac{j!^{1-R}}{k!^{2 \epsilon}}
        \\
        & \geq
        k!^{1 - R - 2 \epsilon}.
    \end{align*}
    Hence for every continuous seminorm $p_R$ in the $\Tensor_R$-topology, we
    have
    \begin{equation*}
     	p_R\left( a_k \star_z b_k \right)
     	\longrightarrow
     	\infty
     	\quad \text{ and } \quad
     	n_R(a_k), n_R(b_k)
     	\longrightarrow
     	0,
    \end{equation*}
    so the star product is not continuous.
\end{example}

%
% An inductive continuity result
%

\subsection{An Inductive Continuity Result}

As already mentioned, we also get continuity via Proposition
\ref{proposition:GuttStarOneLinearFactor} by imposing the
submultiplicativity of the seminorms:
\begin{equation}
    p([\xi, \eta])
    \leq
    p(\xi)p(\eta).
\end{equation}
This is fulfilled for a big class of Lie algebras but, for example,
no longer for general nilpotent ones. In any case, it gives an alternative
proof of the most important part of our Main Theorem and is therefore
given here.
\begin{lemma}
    \label{Lemma:LCAna:PreContinuity2}%
    Let $\lie{g}$ be a locally multiplicatively convex Lie algebra and
    $R \geq1$.  Then if $|z| < 2 \pi$ or $R > 1$ there exists, for $x
    \in \Tensor^{\bullet}(\lie{g}), \eta \in \lie{g}$ of degree at
    most $k$ and each continuous submultiplicative seminorm $p$, a
    constant $c_{z,R}$ only depending on $z$ and $R$ such that the
    following estimate holds:
    \begin{equation}
        \label{LCAna:PreContinuity2}
        p_R(x \ostar_z \eta)
        \leq
        c_{z,R} (k+1)^R p_R(x) q(\eta)
    \end{equation}
\end{lemma}
\begin{proof}
    We have for $\xi_1, \ldots, \xi_k, \eta \in \lie{g}$
    \begin{align*}
        p_R \left(
            \xi_1 \tensor
        \right.
        &
        \left.
             \cdots \tensor \xi_k \ostar_z \eta
        \right)
        =
        p_R \left(
            \sum\limits_{n=0}^k
            \frac{B_n^* z^n}{n! (k-n)!}
            \sum\limits_{\sigma \in S_k}
            \xi_{\sigma(1)} \cdots \xi_{\sigma(k - n)}
            \left(
                \ad_{\xi_{\sigma(k - n + 1)}}
                \circ \cdots \circ
                \ad_{\xi_{\sigma(k)}}
            \right)
            (\eta)
        \right)
        \\
        & =
        \sum\limits_{n=0}^k
        \frac{|B_n^*| |z|^n}{n! (k - n)!}
        \sum\limits_{\sigma \in S_k}
        (k + 1 - n)!^R
        p^{k + 1 - n}
        \left(
            \xi_{\sigma(1)} \cdots \xi_{\sigma(k - n)}
            \left(
                \ad_{\xi_{\sigma(k - n + 1)}}
                \circ \cdots \circ
                \ad_{\xi_{\sigma(k)}}
            \right)
            (\eta)
        \right)
        \\
        & \leq
        (k + 1)^R
        \sum\limits_{n=0}^k
        \frac{|B_n^*| |z|^n}{n!}
        (k - n)!^{R - 1}
        k! p(\xi_1) \cdots p(\xi_k) p(\eta)
        \\
        & =
        (k + 1)^R
        \sum\limits_{n=0}^k
        \frac{|B_n^*| |z|^n}{n!^R}
        \left( \frac{(k - n)! n!}{k!} \right)^{R - 1}
        p_R \left(
         	\xi_1 \tensor \cdots \tensor \xi_k
        \right)
        p(\eta)
        \\
        & \leq
        (k + 1)^R
		p_R \left(
         	\xi_1 \tensor \cdots \tensor \xi_k
        \right)
        p(\eta)
        \sum\limits_{n=0}^k
        \frac{|B_n^*| |z|^n}{n!^R}.
    \end{align*}
    Now if $|z| < 2 \pi$ the sum can be estimated by extending it to a
    series which converges. So we get a constant $c_{z, R}$ depending
    on $R$ and on $z$ such that
    \begin{equation*}
        p_R \left(
            \xi_1
            \tensor \cdots \tensor \xi_k
            \ostar_z
            \eta
        \right)
        \leq
        (k + 1)^R c_{z, R}
		p_R \left(
         	\xi_1 \tensor \cdots \tensor \xi_k
        \right)
        p(\eta).
    \end{equation*}
    On the other hand, if $|z| \geq 2 \pi$ and $R > 1$ we can estimate
    \begin{equation*}
        p_R \left(
            \xi_1
            \tensor \cdots \tensor \xi_k
            \ostar_z
            \eta
        \right)
        \leq
        (k + 1)^R
		p_R \left(
         	\xi_1 \tensor \cdots \tensor \xi_k
        \right)
        p(\eta)
        \left(
            \sum\limits_{n=0}^k
            \frac{|B_n^*|}{n!}
        \right)
        \left(
            \sum\limits_{n=0}^k
            \frac{|z|^n}{n!^{R - 1}}
        \right).
    \end{equation*}
    Again, both series will converge and give constants depending only
    on $z$ and $R$. Hence, we have the estimate on factorizing tensors
    and can extend this to generic tensors of degree at most $k$ by taking
    the infimum as in the proof of Proposition~\ref{Prop:LCAna:Continuity1}.
\end{proof}

In the following, we assume again that either $R > 1$ or $R \geq 1$
and $|z| < 2\pi$ in order the use
Lemma~\ref{Lemma:LCAna:PreContinuity2}.  Now we can give a simpler
proof of Proposition~\ref{Prop:LCAna:Continuity1} for the case of a
locally multiplicatively convex Lie algebra:
\begin{proof}[Alternative Proof of Proposition~\ref{Prop:LCAna:Continuity1}]
    Assume that $\lie{g}$ is now even locally multiplicatively convex.
    We want to replace $\eta$ in the foregoing lemma by an arbitrary
    tensor $y$ of degree at most $\ell$. Let $\eta_1, \ldots,
    \eta_\ell \in \lie{g}$.  On factorizing tensors we get
    \begin{align*}
        p_R \left(
            \xi_1 \tensor \cdots \tensor \xi_k
            \ostar
            \right.
        &
            \left.
            \eta_1 \tensor \cdots \tensor \eta_{\ell}
        \right)
        =
        p_R \left(
         	\frac{1}{\ell!}
         	\sum\limits_{\tau \in S_{\ell}}
            \xi_1 \tensor \cdots \tensor \xi_k
            \ostar
            \eta_{\tau(1)}
            \ostar \cdots \ostar
            \eta_{\tau(\ell)}
        \right)
        \\
        & \leq
        c_{z,R} (k + \ell)^R
        \frac{1}{\ell!}
        \sum\limits_{\tau \in S_{\ell}}
        p_R \left(
            \xi_1 \tensor \cdots \tensor \xi_k
            \ostar
            \eta_{\tau(1)}
            \ostar \cdots \ostar
            \eta_{\tau(\ell - 1)}
        \right)
        p \left( \eta_{\tau(\ell)} \right)
        \\
        & \leq
        c_{z,R}^{\ell} ((k + \ell) \cdots (k + 1))^R
        p_R \left(
            \xi_1 \tensor \cdots \tensor \xi_k
        \right)
        p(\eta_1) \cdots p(\eta_{\ell})
        \\
        & =
        c_{z,R}^{\ell} \left(
        \frac{(k + \ell)!}{k! \ell!}
        \right)^R
        p_R \left(
            \xi_1 \tensor \cdots \tensor \xi_k
        \right)
        p_R \left(
            \eta_1 \tensor \cdots \tensor \eta_{\ell}
        \right)
        \\
        & \leq
        (2^R p)_R \left(
            \xi_1 \tensor \cdots \tensor \xi_k
        \right)
        (2^R c_{z,R} p)_R \left(
            \eta_1 \tensor \cdots \tensor \eta_{\ell}
        \right).
    \end{align*}
    Once again, we have the estimate on factorizing tensors via
    polarization and extend it via the infimum argument on the whole
    tensor algebra, since the estimate depends no longer on the degree
    of the tensors.
\end{proof}

While the above proof is of course much simpler, we had to invest a
slightly stronger assumption compared to the AE Lie algebra case.

%
% Completion and Dependence on $z$
%

\subsection{Dependence on the Formal Parameter}
\label{sec:Holomorphic}

Let us consider the completion $\widehat{\Sym}_R^\bullet(\lie{g})$ of
the symmetric algebra with the Gutt star product $\star_z$.
Exponentials do not belong to this completion, as we can show.
\begin{proposition}
    \label{proposition:NoExponentialsSorry}%
    Let $\xi \in \lie{g}$ and $R \geq 1$, then $\exp(\xi) \not\in
    \widehat{\Sym}_R^\bullet(\lie{g})$, where $\exp(\xi) =
    \sum_{n=0}^{\infty} \frac{\xi^n}{n!}$.
\end{proposition}
\begin{proof}
    Take $p$ a seminorm such that $p(\xi) \neq 0$. Then set $c =
    p(\xi)^{-1}$. For $\xi^n$ the powers in the sense of either the
    usual tensor product, or the symmetric product or the star product
    are the same. So we have for $N \in \mathbb{N}$
    \begin{equation*}
        (cp)_R \left(
        \sum\limits_{n=0}^N
        \frac{c^n}{n!} \xi^n
        \right)
        =
        \sum\limits_{n=0}^N
        \frac{n!^R}{n!}
        c^n
        p_R \left( \xi^n
        \right)
        =
        \sum\limits_{n=0}^N
        n!^{R - 1}
        \geq
        N.
    \end{equation*}
    Hence $(cp)_R (\exp(\xi))$ does not converge.
\end{proof}

Since the formal series converges to the star product on
$\widehat{\Sym}_R^\bullet(\lie{g})$ and all the projections on the
homogeneous components are continuous from Lemma
\ref{proposition:Projections}, we can reinterpret the continuity
result we found in Proposition~\ref{Prop:LCAna:Continuity1}.
\begin{proposition}
    \label{corollary:HolomorphicDependence}%
    Let $R \geq 1$, then for all $x, y \in
    \widehat{\Sym}_R^\bullet(\lie{g})$ the map
    \begin{equation}
        \label{LCAna:Holomorphicity}
        \mathbb{K} \ni z
        \longmapsto
        x \star_z y \in
        \widehat{\Sym}_R^\bullet(\lie{g})
    \end{equation}
    is analytic with (absolutely convergent) Taylor expansion at $z = 0$
    given by Equation~\eqref{Formulas:2MonomialsFormula1}. The collection
    of algebras
    $\left\{ \left( \widehat{\Sym}_R^\bullet(\lie{g}), \star_z \right)
    \right\}_{z \in \mathbb{K}}$ is an entire deformation of the completed
    symmetric tensor algebra $\widehat{\Sym}_R^\bullet(\lie{g})$.
\end{proposition}
\begin{proof}
	The crucial point is that for $x, y \in \widehat{\Sym}_R^\bullet
	(\lie{g})$ and every continuous seminorm $p$ we have an asymptotic
	estimate $q$ such that
	\begin{align*}
		p_R \left( x \star_z y \right)
		& =
		p_R
		\left(
			\sum\limits_{n=0}^{\infty}
			z^n C_n(x,y)
		\right)
		\\
		& =
		\sum\limits_{n=0}^{\infty}
		|z|^n
		p_R( C_n(x, y) )
		\\
		& \leq
		(16 q)_R (x)
        (16 q)_R (y)
		\sum\limits_{n=0}^{\infty}
        \frac{|z|^n n!^{1 - R}}{2 \cdot 8^n},
	\end{align*}
	where we used the fact that the estimate \eqref{LCAna:CnOperators}
	extends to the completion. For $R > 1$, this map is clearly analytic
	and absolutely convergent for all $z \in \mathbb{K}$.
	If $R = 1$, then for every $M \geq 1$ we go back to homogeneous,
	factorizing tensors $x^{(k)}$ and $y^{(\ell)}$ of degree $k$ and $\ell$
	respectively, and have
	\begin{align*}
		M^{n} p_R \left(
			C_n \left( x^{(k)}, y^{(\ell)} \right)
		\right)
		& \leq
		\frac{M^n}{2 \cdot 8^n}
		(16 q)_R \left( x^{(k)} \right)
		(16 q)_R \left( y^{(\ell)} \right)
		\\
		& \leq
		M^{k + \ell}
		(16 q)_R \left( x^{(k)} \right)
		(16 q)_R \left( y^{(\ell)} \right)
		\\
		& =
		(16M q)_R \left( x^{(k)} \right)
		(16M q)_R \left( y^{(\ell)} \right),
	\end{align*}
	where we used that $0 \leq n \leq k + \ell - 1$. The infimum argument
	gives the estimate on all tensors $x,y \in \Tensor_R^{\bullet}(\lie{g})$
	and it extends to the completion such that
	\begin{equation*}
		p_R \left( z^n C_n(x, y) \right)
		\leq
		(16M q)_R(x) (16M q)_R(y)
		\frac{|z|^n}{2 \cdot (8M)^n}
	\end{equation*}
	and hence
	\begin{equation*}
		p_R(x \star_z y)
		\leq
		(16 M q)_R (x)
        (16 M q)_R (y)
		\sum\limits_{n=0}^{\infty}
        \frac{|z|^n}{2 \cdot (8M)^n}.
	\end{equation*}
	So the power series converges for all $z \in \mathbb{K}$ with
	$|z| < 8M$ and converges uniformly if $|z| \leq c M$ for $c < 8$.
	But then, the right hand side of \eqref{LCAna:Holomorphicity}
	converges on all open discs centred around $z = 0$, and it must
	therefore be entire.
\end{proof}

%
% Functoriality and Representations
%

\subsection{Functoriality and Representations}
\label{subsec:FunctorialityRepresentations}

%If we take the formal parameter $z = 1$, the algebra $\Sym_1^\bullet
%(\lie{g})$ with the Gutt star product $\star_G$ is isomorphic to the
%universal enveloping algebra $\algebra{U}(\lie{g})$. Since $\algebra{U}
%(\lie{g})$ and $\Sym^\bullet(\lie{g})$ have universal properties and we
%endowed them with a topology, we can ask whether we get some functorial
%properties with our construction. In other words:
Let $z \in \mathbb{K}$, $\algebra{A}$ an associative, locally convex
algebra and $\phi_z \colon \lie{g}_z \longrightarrow \algebra{A}$ a
continuous Lie algebra homomorphism with respect to the $z$-scaled
Lie bracket. Then we have the commuting diagram
\begin{equation}
    \label{eq:WhatACoolDiagram}
    \begin{tikzpicture}[baseline = (current bounding box.center)]
        \matrix (m)[
        matrix of math nodes,
        row sep=6em,
        column sep=7em,
        ]
        {
          \algebra{U}_R(\lie{g}_z)
          & \Sym_R^\bullet(\lie{g}) \\
          \lie{g}
          & \algebra{A} \\
        };
        \draw
        [-stealth]
        (m-1-1) edge node [above] {$\mathfrak{q}_z^{-1}$}
        (m-1-2) edge node [left] {$\Phi_z$}
        (m-2-2)
        (m-1-2) edge node [right] {$\widetilde{\Phi}_z$}
        (m-2-2)
        (m-2-1) edge node [left] {$\iota_z$}
        (m-1-1)
        (m-2-1) edge node [below] {$\phi_z$}
        (m-2-2);
    \end{tikzpicture}
\end{equation}
from the algebraic theory. A crucial question is now whether the
algebra homomorphisms $\Phi_z$ and $\widetilde{\Phi}_z$ are continuous. This
question is partly answered by the following result:
\begin{proposition}
    \label{Prop:LCAna:Semi-functoriality}%
    Let $\lie{g}$ be an AE-Lie algebra, let $\algebra{A}$ be an
    associative AE-algebra, and let $\phi_z \colon \lie{g}
    \longrightarrow \algebra{A}$ be a continuous Lie algebra
    homomorphism with respect to the $z$-scaled Lie bracket.  If $R
    \geq 0$, then the induced algebra homomorphisms $\Phi_z$ and
    $\widetilde{\Phi}_z$ are continuous.
\end{proposition}
\begin{proof}
    We define an extension of $\widetilde{\Phi}_z$ on the whole tensor
    algebra by
    \begin{equation*}
        \Psi_z \colon
        \Tensor_R^\bullet(\lie{g})
        \longrightarrow
        \algebra{A},
        \quad
        \Psi
        =
        \widetilde{\Phi}_z \circ \Symmetrizer.
    \end{equation*}
    It is clear that if $\Psi_z$ is continuous on factorizing tensors,
    we get the continuity of $\widetilde{\Phi}_z$ and $\Phi_z$ via the
    infimum argument. So let $p$ be a continuous seminorm on
    $\algebra{A}$ with its asymptotic estimate $q$ and $\xi_1, \ldots,
    \xi_n \in \lie{g}$. Since $\phi_z$ is continuous, we find a
    continuous seminorm $r$ on $\lie{g}$ such that for all $\xi \in
    \lie{g}$ we have $q(\phi_z(\xi)) \leq r(\xi)$. Then we have
    \begin{align*}
        p \left(
           \Psi_z
           \left(
               \xi_1 \tensor \cdots \tensor \xi_n
            \right)
        \right)
        & =
        p \left(
            \widetilde{\Phi}_z
            \left(
                \xi_1 \star_z \cdots \star_z \xi_n
            \right)
        \right)
        \\
        & =
        p( \phi_z(\xi_1) \cdots \phi_z(\xi_n) )
        \\
        & \leq
        q( \phi_z(\xi_1) )
        \cdots
        q( \phi_z(\xi_n) )
        \\
        & \leq
        r(\xi_1) \cdots r(\xi_n)
        \\
        & \leq
        r_R(\xi_1 \tensor \cdots \tensor \xi_n),
    \end{align*}
    where the last inequality is true for all $R \geq 0$.
\end{proof}

Our construction fails to be universal since the universal enveloping
algebra endowed with our topology is \emph{not} AE for $R > 0$. This
is can be seen as follows:
\begin{example}
    Take $\xi \in \lie{g}$, then we know that $\xi^{\tensor n} =
    \xi^{\ostar n} = \xi^n$ for $n \in \mathbb{N}$ where we set the
    deformation parameter to $z = 1$. Let $R > 0$ and let $p$ be a
    continuous seminorm in $\lie{g}$ then we find
    \begin{equation}
        p_R(\xi^n)
        =
        n!^R p(\xi)^n
        =
        \frac{n!^R}{c^n} q(\xi)^n
    \end{equation}
    for $c = \frac{p(\xi)}{q(\xi)}$ for a different seminorm $q$ with
    $q(\xi) \neq 0$.  But since the $\frac{n!^R}{c^n}$  always
    diverges for $n \rightarrow \infty$ we do not get an asymptotic
    estimate for $p_R$.
\end{example}
Nevertheless, from Proposition~\ref{Prop:LCAna:Semi-functoriality} we get the
following conclusion:
\begin{corollary}
    \label{Coro:LCAna:ContinuousRepresentations}%
    Let $R \geq 1$ and $\algebra{U}_R(\lie{g}_z)$ the universal
    enveloping algebra with rescaled Lie bracket of an AE-Lie algebra
    $\lie{g}$, then for every continuous representation $\phi_z$ of
    $\lie{g}_z$ into the bounded linear operators $\Bounded(V)$ on a Banach
    space $V$ the induced homomorphism of associative algebras $\Phi_z \colon
    \algebra{U} (\lie{g}_z) \longrightarrow \Bounded(V)$ is continuous.
\end{corollary}
\begin{proof}
    This follows directly from Proposition~\ref{Prop:LCAna:Semi-functoriality}
    and $\Bounded(V)$ being a Banach algebra.
\end{proof}

Now let $\lie{g}, \lie{h}$ be two AE-Lie algebras. We know that a Lie algebra
homomorphism lifts $\phi$ to a unital homomorphism of algebras $\Phi_z$
\begin{equation}
    \label{eq:YetAnotherCoolDiagram}
    \begin{tikzpicture}[baseline = (current bounding box.center)]
        \matrix (m)[
        matrix of math nodes,
        row sep=3em,
        column sep=7em
        ]
        {
          \Sym_R^{\bullet}(\lie{g})
          & \Sym_R^{\bullet}(\lie{h}) \\
          \algebra{U}_R(\lie{g}_z)
          & \algebra{U}_R(\lie{h}_z) \\
          \lie{g}
          & \lie{h} \\
        };
        \draw
        [-stealth]
        (m-1-1) edge node [above] {$\widetilde{\Phi}_z$}
        (m-1-2)
        (m-2-1) edge node [left] {$\mathfrak{q}_z^{-1}$}
        (m-1-1)
        (m-2-1) edge node [above] {$\Phi_z$}
        (m-2-2)
        (m-2-2) edge node [right] {$\mathfrak{q}_z^{-1}$}
        (m-1-2)
        (m-3-1)	edge node [below] {$\phi$}
        (m-3-2)
        (m-3-1)	edge node [left] {$\iota_z$}
        (m-2-1)
        (m-3-2)	edge node [right] {$\iota_z$}
        (m-2-2);
    \end{tikzpicture}
\end{equation}
for all $z \in \mathbb{K}$ since in this case $\phi = \phi_z\colon
\lie{g}_z \longrightarrow \lie{h}_z$ is a Lie algebra morphism for all
$z \in \mathbb{K}$.  If $\phi$ is a continuous Lie algebra
homomorphism, we can ask if $\Phi_z$ will be continuous, too. The
answer is yes and hence our construction is functorial.  For the
proof, we will need the next two lemmas.
\begin{lemma}
    \label{Lemma:LCAna:NStarPrePreSubResult}%
    Let $\lie{g}$ be an AE-Lie algebra, $n \in \mathbb{N}$, $\xi_1, \ldots, \xi_n
    \in \lie{g}$, $i_j \in \{0, \ldots, j\}$, $\forall_{j = 1,
    \ldots, n-1}$ and denote $I = \sum_j i_j$. Then we have the formula
	\begin{align}
		\nonumber
		z^{i_{n-1}}
		&
		C_{i_{n-1}}\left(
			\ldots
			z^{i_2} C_{i_2} \left(
				z^{i_1} C_{i_1} \left(
					\xi_1, \xi_2
				\right)
				,
				\xi_3
			\right)
			\ldots,
			\xi_n
		\right)
		=
		z^I B_{i_{n-1}}^* \cdots B_{i_1}^*
		\\
		& \quad
		\cdot
		\frac{\binom{1}{i_1}}{1!}
		\frac{\binom{2 -i_1}{i_2}}
		{(2-i_1)!}
		\cdots
		\frac{\binom{n-1 - i_1 - \cdots - i_{n-2}}{i_{n-1}}}
		{(n-1 - i_1 - \cdots - i_{n-2})!}
		\sum\limits_{\substack{
			\sigma_1 \in S_{2 - i_1} \\
			\ldots\\
			\sigma_{n-1} \in S_{n-1 - i_1 - \ldots - i_{n-2}}
		}}
		[w_1] \cdots [w_{n-I}],
	\end{align}
	where the expressions $[w_i]$ denote nested Lie brackets in the $\xi_i$.
\end{lemma}
\begin{proof}
	The proof is done by induction and follows directly from
	Formula~\eqref{Formulas:LinearMonomial2} and the bilinearitiy of the $C_n$.
\end{proof}
\begin{lemma}
    \label{Lemma:LCAna:LemmaPreContinuityN}%
    Let $\lie{g}$ be an AE-Lie algebra, $R \geq 1$ and $z \in
    \mathbb{C}$. Then for $p$ a continuous seminorm, $q$ an asymptotic
    estimate for it, $n \in \mathbb{N}$, and all $\xi_1, \ldots, \xi_n
    \in \lie{g}$ the following estimate
    \begin{equation}
        \label{LCAna:LemmaPreContinuityN}
        p_R \left(
            \xi_1 \star_z \cdots \star_z \xi_n
        \right)
        \leq
        c^n n!^R
        q(\xi_1) \cdots q(\xi_n)
    \end{equation}
    holds with $c = 8 \E (|z| + 1)$.
\end{lemma}
\begin{proof}
    For a continuous seminorm $p$ we have
    \begin{align}
        \nonumber
        p_R \left(
            \xi_1 \star_z \cdots \star_z \xi_n
        \right)
        & =
        p_R \Bigg(
        \sum\limits_{\ell = 0}^{n-1}
        \sum\limits_{\substack{
			1 \leq j \leq n-1 \\
			i_j \in \{0, \ldots, j\} \\
			\sum_{j = 1}^{n - 1} i_j = \ell
		}}
		z^{i_{n-1}}
		C_{i_{n-1}}
		\left(
			\ldots z^{i_2} C_{i_2}
			\left(
				z^{i_1} C_{i_1}
				\left( \xi_1, \xi_2 \right)
				, \xi_3
			\right)
			\ldots, \xi_n
		\right)
        \Bigg)
        \\
        \nonumber
        & \leq
        \sum\limits_{\ell = 0}^{n-1}
        (n - \ell)!^R
        \sum\limits_{\substack{
			1 \leq j \leq n-1 \\
			i_j \in \{0, \ldots, j\} \\
			\sum_{j = 1}^{n - 1} i_j = \ell
		}}
        p^{n - \ell} \Big(
		z^{i_{n-1}}
		C_{i_{n-1}}
		\left(
			\ldots z^{i_2} C_{i_2}
			\left(
				z^{i_1} C_{i_1}
				\left( \xi_1, \xi_2 \right)
				, \xi_3
			\right)
			\ldots, \xi_n
		\right)
		\Big)
        \\
        \nonumber
        & \stackrel{(a)}{\leq}
        \sum\limits_{\ell = 0}^{n-1}
        (n - \ell)!^R
        \sum\limits_{\substack{
			1 \leq j \leq n-1 \\
			i_j \in \{0, \ldots, j\} \\
			\sum_{j = 1}^{n - 1} i_j = \ell
		}}
        |z|^\ell
		|B_{i_1}^*| \cdots |B_{i_{n-1}}^*|
		\\
		\nonumber
		& \quad
		\cdot
		\binom{1}{i_1} \binom{2 - i_1}{i_2}
		\cdots
		\binom{n-1 - i_1 - \cdots - i_{n-2}}{i_{n-1}}
        q \left(\xi_1\right) \cdots q \left(\xi_n\right)
        \\
        \label{LCAna:PreContinuityIntermediateN}
        & \stackrel{(b)}{\leq}
        \sum\limits_{\ell = 0}^{n-1}
        (n - \ell)!^R
        \sum\limits_{\substack{
			1 \leq j \leq n-1 \\
			i_j \in \{0, \ldots, j\} \\
			\sum_{j = 1}^{n - 1} i_j = \ell
		}}
        |z|^{i_{n-1}} \cdots |z|^{i_1}
		1^{i_1} 2^{i_2}
        \cdots (n-1)^{i_{n-1}}
        q(\xi_1) \cdots q(\xi_n).
    \end{align}
    In ($a$), we used Lemma~\ref{Lemma:LCAna:NStarPrePreSubResult} and used the
    AE-property. Then the inverse factorials cancel with the sums over the
    permutations. In ($b$), we used the estimate
    \begin{equation*}
    	|B_{i_j}^*| \binom{j - i_1 - \ldots - i_{j-1}}{i_j}
    	\leq
    	i_j! \binom{j - i_1 - \ldots - i_{j-1}}{i_j}
    	=
    	\frac{(j - i_1 - \ldots - i_{j-1})!}
    	{(j - i_1 - \ldots - i_j)!}
    	\leq
    	j^{i_j}
    \end{equation*}
    for all $j = 1, \ldots, n-1$. Now, we estimate the number of terms in the sum
    and get
    \begin{equation*}
        \sum\limits_{\ell = 0}^{n-1}
        \sum\limits_{\substack{
            1 \leq j \leq n-1 \\
            i_j \in \{0, \ldots, j\} \\
            \sum_{j = 1}^{n - 1} i_j = \ell
          }}
        1
        \stackrel{(a)}{\leq}
        \sum\limits_{\ell = 0}^{n-1}
        \binom{n - 1 + \ell - 1}{\ell - 1}
        \stackrel{(b)}{\leq}
        2^{2n}.
    \end{equation*}
    In ($a$) the estimate for the big sum is that for every $j = 1,
    \ldots, n-1$ we have $i_j \in \{0, 1, \ldots, n-1\}$ and the sum
    of all the $i_j$ is $\ell$. If we forget about all other
    restrictions the number of summands equals the sum of ways to
    distribute $\ell$ items on $n-1$ places, which is given by
    $\binom{n - 1 + \ell - 1}{\ell - 1}$. In ($b$) we use
    \begin{equation*}
        \binom{n - 1 + \ell - 1}{\ell - 1} \leq \binom{2 n}{\ell - 1}
    \end{equation*}
    with the binomial coefficient being zero for $\ell = 0$.  By using
    the fact that $|B_m^*| \leq m!$ for all $m \in \mathbb{N}$ and
    grouping together the powers of $|z|$, we get from
    Inequality~\eqref{LCAna:PreContinuityIntermediateN}
    \begin{align*}
        p_R \left(
            \xi_1 \star_z \cdots \star_z \xi_n
        \right)
        & \stackrel{(a)}{\leq}
        \sum\limits_{\ell = 0}^{n-1}
        (n - \ell)!^R
        \sum\limits_{\substack{
			1 \leq j \leq n-1 \\
			i_j \in \{0, \ldots, j\} \\
			\sum_{j = 1}^{n - 1} i_j = \ell
		}}
        |z|^{\ell}
        n^{\ell}
        q(\xi_1) \cdots q(\xi_n)
        \\
        &\stackrel{(b)}{\leq}
        \sum\limits_{\ell = 0}^{n-1}
        (n - \ell)!^R
        |z|^{\ell}
        \E^n 2^n \ell!
        q(\xi_1) \cdots q(\xi_n)
        \sum\limits_{\substack{
			1 \leq j \leq n-1 \\
			i_j \in \{0, \ldots, j\} \\
			\sum_{j = 1}^{n - 1} i_j = \ell
              }}
              1
        \\
        &\stackrel{(c)}{\leq}
        n!^R
        (|z| + 1)^n (8 \E)^n
        q(\xi_1) \cdots q(\xi_n).
    \end{align*}
    We used
     \begin{equation*}
        1^{i_1} \cdots (n-1)^{i_{n-1}}
        \leq
        n^{\ell},
    \end{equation*}
    since $\sum_{j=1}^{n-1} i_j = \ell$ in ($a$) and
    $n^{\ell} \leq \E^n \frac{n!}{(n-\ell)!} = \E^n \binom{n} {\ell}
    \ell!  \leq \E^n 2^n \ell!$ in ($b$). The last step ($c$) is just $(n -
    \ell)! \ell!  \leq n!$, the estimate for the sum, and $(|z| + 1)
    \geq |z|$, which finishes the proof.
\end{proof}
\begin{proposition}
    \label{proposition:Functoriality}%
    Let $R \geq 1$, let $\lie{g}, \lie{h}$ be AE-Lie algebras and let
    $\phi \colon \lie{g} \longrightarrow \lie{h}$ be a continuous
    homomorphism between them. Then it lifts to a continuous unital
    homomorphism of locally convex algebras $\Phi_z \colon
    \algebra{U}_R(\lie{g}_z) \longrightarrow \algebra{U}_R(\lie{h}_z)$
    for all $z \in \mathbb{K}$.
\end{proposition}
\begin{proof}
	First, if $\phi \colon \lie{g} \longrightarrow \lie{h}$ is continuous,
	then for every continuous seminorm $q$ on $\lie{h}$, we have a continuous
	seminorm $r$ on $\lie{g}$ such that for all $\xi \in \lie{g}$
	\begin{equation*}
		q\left( \phi(\xi) \right)
		\leq
		r(\xi).
	\end{equation*}
	Second, we define $\Psi_z$ by
	\begin{equation*}
		\Psi_z \colon
		\Tensor_R^{\bullet}(\lie{g})
		\longrightarrow
		\Sym_R^{\bullet}(\lie{h})
		, \quad
		\Psi_z
		=
		\widetilde{\Phi}_z \circ
		\Symmetrizer
	\end{equation*}
	as before. Clearly, $\Phi_z$ and $\widetilde{\Phi}_z$ will be
    continuous if $\Psi_z$ is continuous.  From this, we get for a
    seminorm $p$ on $\lie{h}$, an asymptotic estimate $q$ for it,
    and $\xi_1, \ldots, \xi_n$
	\begin{align*}
		p_R\left(
			\Psi_z \left(
				\xi_1 \tensor \ldots \tensor \xi_n
			\right)
		\right)
		& =
		p_R \left(
			\phi \left( \xi_1 \right)
			\star_z \ldots \star_z
			\phi \left( \xi_1 \right)
		\right)
		\\
		& \ot{(a)}{\leq}
		c^n n!^R
		q \left( \phi \left( \xi_1 \right) \right)
		\ldots
		q \left( \phi \left( \xi_n \right) \right)
		\\
		& \ot{(b)}{\leq}
		c^n n!^R
		r \left( \xi_1 \right)
		\ldots
		r \left( \xi_n \right)
		\\
		& =
		(c r)_R\left(
			\xi_1 \tensor \ldots \tensor \xi_n
		\right).
	\end{align*}
	Again, we use the infimum argument and we have the estimate on all tensors
	in $\Tensor_R^{\bullet}(\lie{g})$. It extends to the completion and the
	statement is proven.
\end{proof}

%
% Nilpotent Lie Algebras
%

\section{Nilpotent Lie Algebras}
\label{sec:Nilpotent}

Let us now consider nilpotent locally convex Lie algebras. Our results are
still valid in this case but we can make some more
observations. On one hand, things should not change drastically for
nilpotent Lie algebras; in fact, Example~\ref{Ex:LCAna:HeisenbergAlgebra}
shows that even for nilpotent Lie algebras the star product is not
continuous for $R < 1$. Thus, there is no reason to expect much larger
algebras and completions. On the other hand, the Weyl algebra studied in
\cite{waldmann:2014a} is a quotient of the Heisenberg algebra and has a
continuous product for $R \geq \frac 1 2$. The quotient procedure must
therefore have some influence on the estimates. Finally, the fact that
exponentials are not in $\widehat{\Sym}_1^{\bullet}(\lie{g})$ is not
unexpected. If so, Equation~\eqref{eq:DrinfeldIsGuttStar} would mean that
could give some sense to $\bch{\xi}{\eta}$ for all $\xi, \eta$ from some
arbitrary Lie algebra $\lie{g}$, which would be surprising. In the
nilpotent case, this is no longer the case since the Baker-Campbell-Hausdorff
series converges globally. So it would be nice to have something more than
$R \geq 1$.

%
% A Projective Limit
%

\subsection{A Projective Limit}
\label{subsec:ProjectiveLimit}

In the following we show that the Gutt star product is continuous in
the projective limit $R \longrightarrow R^-$ that and the exponential
actually belongs to the completion.
\begin{proposition}
    \label{proposition:Nilpot:ProjLimit}%
    Let $\lie{g}$ be a nilpotent locally convex Lie algebra with
    continuous Lie bracket and $N \in \mathbb{N}$ such that $N + 1$
    Lie brackets vanish.
    \begin{propositionlist}
    \item \label{item:Nilpot:CnOperators} If $0 \leq R < 1$, the
        $C_n$-operators are continuous and fulfil the estimate
     	\begin{equation}
            \label{eq:Nilpot:CnOperators}
            p_R \left(
                C_n (x, y)
            \right)
            \leq
            \frac{1}{2 \cdot 8^n}
            (32 \E q)_{R + \epsilon}(x)
            (32 \E q)_{R + \epsilon}(y),
     	\end{equation}
     	for all $x, y \in \Sym_R^{\bullet}(\lie{g})$, where $p$ is a
        continuous seminorm, $q$ an asymptotic estimate for $p$, and
        $\epsilon = \frac{N - 1}{N}(1 - R)$.
    \item \label{item_Nilpot:SEinsMinus} The Gutt star product
        $\star_z$ is continuous for the locally convex projective
        limit $\Sym_{1^-}^\bullet(\lie{g})$ and we have
     	\begin{equation}
            \label{eq:Nilpot:Continuity}
            p_R \left(
                x \star_z y
            \right)
            \leq
            (c q)_{R + \epsilon}(x)
            (c q)_{R + \epsilon}(y)
     	\end{equation}
        with $c = 32 \E (|z| +1)$.  Hence it extends continuously to
        $\widehat{\Sym}_R^{\bullet} (\lie{g})$, where it coincides
        with the formal series.
    \end{propositionlist}
\end{proposition}
\begin{proof}
    In this proof, we use $\star_z$ on the whole tensor algebra and
    compute the estimate for factorizing tensors $\xi_1 \tensor \cdots
    \tensor \xi_k$ and $\eta_1 \tensor \cdots \tensor \eta_{\ell}$ but
    in this case we get restrictions for the values of $n$. Recall
    that $C_n \left( \xi^{\tensor k}, \eta^{\tensor \ell} \right)$ has
    $n$ brackets it has degree $k + \ell - n$. We get therefore
    \begin{equation*}
     	(k + \ell - n) N
     	\geq
     	k + \ell
     	\quad
     	\Longleftrightarrow
     	\quad
     	n
     	\leq
     	(k + \ell)
     	\frac{N - 1}{N}.
    \end{equation*}
    For all $n$ that violate this condition, $C_n \left( \xi_1 \tensor \cdots
    \tensor \xi_k, \eta_1 \tensor \cdots \tensor \eta_{\ell} \right) = 0$.
    Hence, we get bounds for $n!^{1-R}$ in Equation~\eqref{LCAna:CnOperators}:
    set $\delta = \frac{N - 1}{N}$ and also denote a factorial where we have
    non-integers, meaning the gamma function. We get
    \begin{align*}
        n!^{1-R}
        & \leq
        (\delta (k + \ell)!)^{1 - R}
        \\
        & \leq
        (\delta (k + \ell))^{(1 - R) \delta (k + \ell)}
        \\
        & \leq
        (k + \ell)^{(1 - R) \delta (k + \ell)}
        \\
        & =
        \left(
            (k + \ell)^{(k + \ell)}
        \right)^{(1 - R) \delta}
        \\
        & \leq
        \left(
            \E^{k + \ell} 2^{k + \ell} k! \ell!
        \right)^{(1-R) \delta}
        \\
        & =
        \left( (2 \E)^{\delta (1-R)} \right)^{k + \ell}
        k!^{\epsilon} \ell!^{\epsilon},
    \end{align*}
    using $\epsilon = \delta (1 - R)$. Then
    \begin{align*}
        &p_R \left(
         	C_n \left(
         		\xi_1 \tensor \cdots \tensor \xi_k,
         		\eta_1 \tensor \cdots \tensor \eta_{\ell}
         	\right)
        \right)
        \\
        &\quad\leq
        \frac{
         	\left(
         		(2 \E)^{\delta (1 - R)}
         	\right)^{k + \ell}
        	k!^{\epsilon}
        	\ell!^{\epsilon}
        }{2 \cdot 8^n}
        (16 q)_R \left(
         	\xi_1 \tensor \cdots \tensor \xi_k
        \right)
        (16 q)_R \left(
         	\eta_1 \tensor \cdots \tensor \eta_{\ell}
        \right)
        \\
        &\quad\leq
        \frac{1}{2 \cdot 8^n}
        (c q)_{R + \epsilon}
        \left(
         	\xi_1 \tensor \cdots \tensor \xi_k
        \right)
        (c q)_{R + \epsilon}
        \left(
         	\eta_1 \tensor \cdots \tensor \eta_{\ell}
        \right)
    \end{align*}
    with $c = 16 (2 \E)^{\delta (1 - R)} \leq 32 \E$.  We then get the
    estimate on all tensors, extend it to the completion and get the
    estimate~\eqref{eq:Nilpot:CnOperators}. Recall, that for every $R
    < 1$ we also have $R + \epsilon < 1$ with the $\epsilon =
    \delta(1-R)$ from above. Iterating this continuity estimate, we
    get arbitrarily close to $1$ and it is not possible to repeat this
    process an arbitrary number of times and stop at some value
    strictly less than $1$. The proof of
    Equation~\eqref{eq:Nilpot:Continuity} is done analogously to the
    proof of part \refitem{item:LCAna:Continuity1} of
    Proposition~\ref{Prop:LCAna:Continuity1}.
\end{proof}
\begin{remark}
    \label{remark:SubmultiplicativeMakesThingsEasier}%
    Assuming even submultiplicativity of the seminorms one ends up
    with the same result by using the easier formula
    \eqref{Formulas:LinearMonomial2}.
\end{remark}
Of course, the projective limit case gives us a slightly bigger
completion. We immediately get the following result:
\begin{corollary}
    \label{corollary:NilpotentCase}%
    Let $\lie{g}$ be a nilpotent, locally convex Lie algebra.
    \begin{corollarylist}
    \item \label{item:NilpotentHasExp} For $\xi \in \lie{g}$, the we
        have $\exp(\xi) \in \widehat{\Sym}_{1^-}^\bullet(\lie{g})$.
    \item \label{item:NilpotentExpGivesBCH} For $\xi, \eta \in
        \lie{g}$ and $z \neq 0$ we have $\exp(\xi) \ostar_z
        \exp(\eta) = \exp \left(\frac 1 z \bch{z \xi}{z \eta}
        \right)$.
    \item \label{item:NipotentOneParameterGroups} For $s,t \in
        \mathbb{K}$ and $\xi \in \lie{g}$ we have $\exp(t \xi)
        \ostar_z \exp(s \xi) = \exp ((t + s) \xi)$.
    \end{corollarylist}
\end{corollary}
\begin{proof}
    For the first part, recall that the completion of the projective
    limit $1^-$ contains all those series $x = \sum_{n=0}^\infty x_n$
    with $x_n \in \Sym^n(\lie{g})$ such that
    \begin{equation*}
        \sum\limits_{n=0}^{\infty}
        p^n(x_n) n!^{1 - \epsilon}
        <
        \infty
    \end{equation*}
    for all continuous seminorms $p$ and all $\epsilon > 0$.  This is
    the case for the exponential series of $t \xi$ for $t \in
    \mathbb{K}$ and $\xi \in \lie{g}$. The second part follows from
    the fact that all the projections $\pi_n$ onto the homogeneous
    subspaces $\Sym_{\pi}^n$ are continuous. The third part is then a
    direct consequence of the second.
\end{proof}

%
% A bit Functoriality also in this case
%

\subsection{Functoriality and Representations}
\label{subsec:NilpotentFunctorialityRepresentations}

In the nilpotent case we get the same results for the extension of
maps from $\lie{g}$ into
associative AE-algebras and the same functorial properties in this case.
We just need to adapt Lemma~\ref{Lemma:LCAna:LemmaPreContinuityN}:
\begin{lemma}
    \label{Lemma:Nilpot:LemmaPreContinuityN}%
    Let $\lie{g}$ be locally convex nilpotent Lie algebra, $0 \leq R <
    1$ and $z \in \mathbb{C}$. Then for a continuous seminorm $p$ with
    an asymptotic estimate $p$, $n \in \mathbb{N}$ and all $\xi_1,
    \ldots, \xi_n \in \lie{g}$ the following estimate
    \begin{equation}
        \label{Nilpot:LemmaPreContinuityN}
        p_R \left(
            \xi_1 \star_z \cdots \star_z \xi_n
        \right)
        \leq
        c^n n!^{R + \epsilon}
        q^n(\xi_1 \tensor \cdots \tensor \xi_n)
    \end{equation}
    holds with $c = 16 \E^2 (|z| + 1)$ and $\epsilon = \frac{N-1}{N}(1
    - R)$.  The estimate is locally uniform in $z$.
\end{lemma}
\begin{proof}
    Take $R < 1$ and consider the estimate
    \eqref{LCAna:PreContinuityIntermediateN} in the proof of
    Lemma~\ref{Lemma:LCAna:LemmaPreContinuityN}.  We know that, since
    we can have at most $N$ brackets, also the values for $\ell$ are
    restricted to
    \begin{equation*}
        \ell
        \leq
        \frac{N-1}{N} n
        =
        \delta n.
    \end{equation*}
    Using the same steps of the proof of
    Lemma~\ref{Lemma:LCAna:LemmaPreContinuityN}
    we get
    \begin{align*}
        p_R \left(
            \xi_1 \star_z \cdots \star_z \xi_n
        \right)
        & \leq
        \sum\limits_{\ell = 0}^{\delta n}
        (n - \ell)!^R
        \sum\limits_{\substack{
			1 \leq j \leq n-1 \\
			i_j \in \{0, \ldots, j\} \\
			\sum_{j = 1}^{n - 1} i_j = \ell
		}}
        |z|^{\ell} n^{\ell}
        q(\xi_1) \cdots q(\xi_n)
        \\
        & \leq
        \sum\limits_{\ell = 0}^{\delta n}
        (n - \ell)!^R
        \sum\limits_{\substack{
			1 \leq j \leq n-1 \\
			i_j \in \{0, \ldots, j\} \\
			\sum_{j = 1}^{n - 1} i_j = \ell
		}}
        |z|^{\ell} (2 \E)^n \ell!
        q(\xi_1) \cdots q(\xi_n)
        \\
        & \leq
        (2 \E)^n (|z| + 1)^n
        q(\xi_1) \cdots q(\xi_n)
        \sum\limits_{\ell = 0}^{\delta n}
        (n - \ell)!^R \ell!
        \binom{n + \ell - 2}{\ell - 1}.
    \end{align*}
    We have the inequality
    \begin{equation*}
        \ell!
        =
        \ell!^R
        \ell!^{1-R}
        \leq
        \ell!^R
        \left(
            (\delta n)^{\delta n}
        \right)^{1-R}
        \leq
        \ell!^R
        n^{\delta n (1-R)}
        \leq
        \ell!^R
        n!^{\delta (1 - R)}
        \E^{\delta n (1 - R)}
    \end{equation*}
    and together with $\ell!^R (n - \ell)!^R \leq n!^R$ this gives
    \begin{align*}
        p_R \left(
            \xi_1 \star_z \cdots \star_z \xi_n
        \right)
        & \leq
        (2 \E)^n (|z| + 1)^n
        n!^R n!^{\delta (1 - R)}
        q(\xi_1) \cdots q(\xi_n)
        \sum\limits_{\ell = 0}^{\delta n}
        \binom{n + \ell - 2}{\ell - 1}
        e^{\delta n (1 - R)}
        \\
        & \leq
        (2 \E)^n (|z| + 1)^n
        \left(\E^{(1-R) \delta}\right)^n
        4^n n!^{R + \epsilon}
        q(\xi_1) \cdots q(\xi_n),
    \end{align*}
    with $\epsilon = \delta (1-R)$. It is clear that for all $R < 1$
    we have $R + \epsilon < 1$. To complete the proof we set
    \begin{equation*}
    	c
    	=
    	8 \E (|z|+1) \E^{(1-R)\delta}
    	\leq
    	16 \E^2 (|z|+1),
    \end{equation*}
    and notice that the estimate is locally uniform in $z$, even
    though it is not uniform in $z$.
\end{proof}
\begin{proposition}
    \label{proposition:FunctorialityNilpotent}%
    Let $\lie{g}$ be a nilpotent locally convex Lie algebra with
    continuous Lie bracket. Then the statements of
    Proposition~\ref{Prop:LCAna:Semi-functoriality},
    Corollary~\ref{Coro:LCAna:ContinuousRepresentations}, and
    Proposition~\ref{proposition:Functoriality} hold for the
    projective limit $\Sym_{1^-}^{\bullet}(\lie{g})$, too.
\end{proposition}

%
% The Link to the Weyl Algebra
%

\subsection{The Link to the Weyl Algebra}
\label{subsec:LinkToWeylAlgebra}

In this section we aim to discuss the link to the Weyl algebra from
\cite{waldmann:2014a}. For simplicity, we consider the easiest case of
the Weyl algebra with two generators. Recall that the Weyl algebra is
a quotient of the enveloping algebra of the Heisenberg algebra
$\lie{h}$ which one gets from dividing out its center. So let $c \in
\mathbb{K}$ and we have a projection
\begin{equation}
    \label{Nilpot:WeylProjection}
    \pi \colon
    \Sym_R^\bullet(\lie{h})
    \longrightarrow
    \mathcal{W}_R(\lie{h})
    =
    \left(
    	\frac{\Sym_R^\bullet(\lie{h})}
    	{\langle E - c \Unit \rangle}
    \right)
\end{equation}
\begin{proposition}
    \label{proposition:ProjectionWeylContinuous}%
    The projection $\pi$ is continuous for $R \geq 0$.
\end{proposition}
\begin{proof}
    We extend $\pi$ to the whole tensor algebra by symmetrizing
    beforehand. Let then $p$ be a continuous seminorm on $\lie{h}$, $k,
    \ell, m \in \mathbb{N}_0$. We have
    \begin{align*}
        p_R(\pi (
        	Q^{\tensor k} \tensor
        	P^{\tensor \ell} \tensor
        	E^{\tensor m}
        ) )
        & =
        p_R( Q^k P^{\ell} c^m )
        \\
        & =
        |c|^m (k + \ell)!^R
        p^{k + \ell}(Q^k P^{\ell})
        \\
        & \leq
        (|c| + 1)^{k + \ell + m}
        (k + \ell + m)!^R
        p(Q)^k p(P)^{\ell} p(E)^m
        \\
        & =
        ((|c| + 1) p)_R
        (Q^{\tensor k} \tensor
        P^{\tensor \ell} \tensor
        E^{\tensor m}).
    \end{align*}
    Then we do the usual infimum argument and have the result on
    arbitrary tensors again.
\end{proof}

To establish the link to the continuity results of the Weyl algebra,
we need $\pi \circ \ostar_z$ to be continuous for $R \geq \frac
1 2$.
\begin{proposition}
    \label{proposition:ContinuousProductInWeyl}%
    Let $R \geq \frac{1}{2}$ and $\pi$ the projection from Equation
    \eqref{Nilpot:WeylProjection}. Then the map $\pi \circ
    \ostar_z$ is continuous.
\end{proposition}
\begin{proof}
    Since we are in finite dimensions, we can choose a
    submultiplicative norm $p$ with $p(Q) = p(P) = p(E)$ without
    restriction. Moreover, let $k, k', \ell, \ell', m, m' \in
    \mathbb{N}_0$. Then we have to get an estimate for $p_R \left(
        \pi\left( Q^k P^{\ell} E^m \ostar_z Q^{k'} P^{\ell'} E^{m'}
        \right) \right)$.  If we calculate the star product
    explicitly, we see, that we only get Lie brackets where we have
    $P$'s and $Q$'s. Let $r = k + \ell + m$ and $s = k' + \ell' + m'$,
    then we can actually simplify the calculations by
    \begin{align*}
        p_R \left(
        \pi(Q^k P^{\ell} E^m
        \ostar_z    Q^{k'} P^{\ell'} E^{m'}
        ) \right)
        & =
        (p_R \circ \pi) \left(
        \sum\limits_{n=0}^{r + s - 1}
        z^n C_n(Q^k P^{\ell} E^m,
        Q^{k'} P^{\ell'} E^{m'})
        \right)
        \\
        & \leq
        \sum\limits_{n=0}^{r + s - 1}
        |z|^n
        (p_R \circ \pi) \left(
        C_n(Q^k P^{\ell} E^m,
        Q^{k'} P^{\ell'} E^{m'})
        \right)
        \\
        & \leq
        \sum\limits_{n=0}^{r + s - 1}
        |z|^n
        (p_R \circ \pi) \left(
        C_n(Q^r, P^s)
        \right)
        \\
        & =
        \sum\limits_{n=0}^{r + s - 1}
        |z|^n
        \frac{r! s!}{(r-n)! (s-n)! n!}
        (p_R \circ \pi) \left(
        Q^{r-n} P^{s-n} E^n
        \right)
        \\
        & =
        \sum\limits_{n=0}^{r + s - 1}
        |z|^n |c|^n
        \frac{r! s!}{(r-n)! (s-n)! n!}
        p_R \left(
        Q^{r-n} P^{s-n}
        \right)
        \\
        & \leq
        \sum\limits_{n=0}^{r + s - 1}
        |z|^n |c|^n
        \frac{r! s!}{(r-n)! (s-n)! n!}
        \frac{(r + s - 2n)!^R}{r!^R s!^R}
        p_R \left(Q^{\tensor r} \right)
        p_R \left(P^{\tensor s} \right)
        \\
        & \leq
        \sum\limits_{n=0}^{r + s - 1}
        |z|^n |c|^n
        \binom{r}{n} \binom{s}{n}
        \frac{(r + s - 2n)!^R n!}{r!^R s!^R}
        p_R \left(Q^{\tensor r} \right)
        p_R \left(P^{\tensor s} \right)
        \\
        & \leq
        \sum\limits_{n=0}^{r + s - 1}
        |z|^n |c|^n
        \binom{r}{n} \binom{s}{n}
        \frac{(r + s - 2n)!^R n!}{r!^R s!^R}
        p_R \left(Q^{\tensor r} \right)
        p_R \left(P^{\tensor s} \right)
        \\
        & \ot{(a)}{\leq}
        \sum\limits_{n=0}^{r + s - 1}
        |z|^n |c|^n
        \binom{r}{n} \binom{s}{n}
        \binom{r + s}{s}^R
        \binom{r + s}{2n}^{-R}
        p_R \left(Q^{\tensor r} \right)
        p_R \left(P^{\tensor s} \right)
        \\
        & \leq
        \sum\limits_{n=0}^{r + s - 1}
        (|z| + 1)^n (|c| + 1)^n
        4^{r + s}
        p_R \left(Q^{\tensor r} \right)
        p_R \left(P^{\tensor s} \right)
        \\
        & \ot{(b)}{ \leq }
        \underbrace{
        (8 (|z| + 1) (|c| + 1))^{r + s}
        }_{ = \tilde{c}^{r + s}}
        p_R \left(Q^{\tensor r} \right)
        p_R \left(P^{\tensor s} \right)
        \\
        & =
        (\tilde{c} p)_R \left(Q^{\tensor r} \right)
        (\tilde{c} p)_R \left(P^{\tensor s} \right)
        \\
        & \ot{(c)}{ = }
        (\tilde{c} p)_R \left(
        Q^{\tensor k} \tensor
        P^{\tensor \ell} \tensor
        E^{\tensor m} \right)
        (\tilde{c} p)_R \left(
        Q^{\tensor k'} \tensor
        P^{\tensor \ell'} \tensor
        E^{\tensor m'} \right),
    \end{align*}
    where in we rearranged the factorials in ($a$) and used $R \geq
    \frac{1}{2}$. The estimates ($b$) are the standard binomial
    coefficient estimates. In ($c$) we used $p(Q) = p(P) = p(E)$. Now
    we just use
    \begin{equation*}
        \left(
        	Q^{\tensor k} \tensor
        	P^{\tensor \ell} \tensor
        	E^{\tensor m}
        \right)
        \ostar_z
        \left(
        	Q^{\tensor k'} \tensor
        	P^{\tensor \ell'} \tensor
        	E^{\tensor m'}
        \right)
        =
        Q^k P^{\ell} E^m
        \ostar_z
        Q^{k'} P^{\ell'} E^{m'}
    \end{equation*}
    and the
    infimum argument to expand this estimate to all tensors.
    This concludes the proof.
\end{proof}

%
% The Hopf Algebra Structure
%

\section{The Hopf Algebra Structure}
\label{sec:Hopf}

In a last step we investigate the continuity of the Hopf algebra
structure maps on $\algebra{U}_R(\lie{g}_z)$. Pulling back the
coproduct $\Delta_z$ and the antipode $S_z$ from
$\algebra{U}_R(\lie{g}_z)$ to $\Sym^\bullet(\lie{g})$ we get a
coproduct $\Delta$ and an antipode $S$ with respect to $\star_z$.  It
is well-known that $\Delta$ and $S$ are independent of $z$ and
coincide with the classical shuffle coproduct and antipode
which on the symmetric algebra can be written as follows:
\begin{lemma}
    \label{Thm:Hopf:Formulas}%
    For $\xi_1, \ldots, \xi_n \in \lie{g}$ we have the identities
    \begin{equation}
        \label{Hopf:AntipodeFormula}
        S(\xi_1 \tensor \cdots \tensor \xi_n)
        =
        (-1)^n
        \xi_1 \cdots \xi_n
    \end{equation}
    and
    \begin{equation}
        \label{Hopf:CoproductFormula}
        \ocoproduct(\xi_1 \tensor \cdots \tensor \xi_n)
        =
        \sum\limits_{
        	I \subseteq
        	\{1, \ldots, n\}
        }
        \xi_I
        \tensor
        \xi_1 \cdots
        \widehat{\xi_I}
        \cdots \xi_n
    \end{equation}
    where $\xi_I$ denotes the symmetric tensor product of all $\xi_i$
    with $i \in I$ and $\widehat{\xi_I}$ means that the $\xi_i$ with
    $i \in I$ are left out.
\end{lemma}

We need a topology on the tensor product in
\eqref{Hopf:CoproductFormula}, for which we take again the projective
tensor product. The continuity of the two maps is now easy to
prove. In fact, it does not refer to the Lie structure at all and
holds in full generality:
\begin{proposition}
    \label{Prop:Hopf:CoproductContinuity}%
    Let $V$ be a locally convex vector space and let $R \geq 0$. For
    every continuous seminorm $p$ and all $x \in
    \widehat{\Sym}_R^\bullet(V)$ the following estimates hold:
    \begin{equation}
        \label{Hopf:AntipodeContinuity}
        p_R \left(S(x)\right)
        \leq
        p_R (x)
    \end{equation}
    and
    \begin{equation}
        \label{Hopf:CoproductContinuity}
        (p_R \tensor p_R)
        \left(\ocoproduct(x)\right)
        \leq
        (2 p)_R (x).
    \end{equation}
\end{proposition}
\begin{proof}
    We use the extension to the whole tensor algebra by symmetrizing
    beforehand. Inequality \eqref{Hopf:AntipodeContinuity} is clear
    on factorizing tensors and extends to all tensors by the infimum
    argument.  To get the estimate~\eqref{Hopf:CoproductContinuity},
    we compute it on factorizing tensors:
    \begin{align*}
        (p_R \tensor p_R)
        \left(
            \ocoproduct\left(
                \xi_1 \tensor \cdots \tensor \xi_n
            \right)
        \right)
        &=
        (p_R \tensor p_R)
        \left(
            \sum\limits_{
              I \subseteq
              \{1, \ldots, n\}
            }
            \xi_I
            \tensor
            \xi_1 \cdots
            \widehat{\xi_I}
            \cdots \xi_n
        \right)
        \\
        &\leq
        \sum\limits_{
          I \subseteq
          \{1, \ldots, n\}
        }
        |I|!^R (n - |I|)!^R
        p^{|I|} \left( \xi_I \right)
        p^{n - |I|}
        \left(
            \xi_1 \cdots \widehat{\xi_I} \cdots \xi_n
        \right)
        \\
        &\leq
        \sum\limits_{
          I \subseteq
          \{1, \ldots, n\}
        }
        |I|!^R (n - |I|)!^R
        p(\xi_1) \cdots p(\xi_n)
        \\
        &\leq
        \sum\limits_{
        	I \subseteq
        	\{1, \ldots, n\}
        }
        n!^R
        p(\xi_1) \cdots p(\xi_n)
        \\
        & =
        2^n n!^R
        p(\xi_1) \cdots p(\xi_n)
        \\
        &=
        (2p)_R \left(
            \xi_1 \tensor \cdots \tensor \xi_n
        \right).
    \end{align*}
    This extends to all tensors by the infimum argument and restricts
    to the symmetric algebra afterwards.
\end{proof}

Since the continuity of the unit and the counit $\epsilon = \pi_0$ is
clear by the definition of our topology, we have the following result.
\begin{proposition}
    \label{Prop:Hopf:ContinuousHopf}%
    Let $\lie{g}$ be an AE-Lie algebra and $z \in \mathbb{C}$. Then,
    if $R \geq 1$, $\widehat{\Sym}_R^\bullet (\lie{g})$ is a
    topological Hopf algebra. The same holds for
    $\widehat{\Sym}_{1^-}^{\bullet}(\lie{g})$, if $\lie{g}$ is a
    nilpotent locally convex Lie algebra with continuous Lie bracket.
\end{proposition}

%
% Outlook and Open Questions
%

\section{Outlook and Open Questions}
\label{sec:outlook}

The two main theorems allow us to make some observations, which also
go beyond deformation quantization. We found a locally convex topology
on the universal enveloping algebra and encountered a special class of
AE(-Lie) algebras. It is worth looking at those two issues more
closely.

\subsection{Asymptotic Estimate Algebras}

The term asymptotic estimate has, to the best of our knowledge, first
been used by Boseck, Czichowski and Rudolph in
\cite{boseck.czichowski.rudolph:1981a}. Their definition of an
AE-algebra differs from ours since not just one but a whole series of
asymptotic estimates has to exist for every seminorm. This series has
to fulfil two technical properties. This is not the case in our
definition, which is, in general, weaker.

Furthermore, in \cite{gloeckner.neeb:2012a}, Glöckner and Neeb use a
property to which they referred as ($*$) for associative
algebras. This was then used in
\cite{bogfjellmo.dahmen.schmedig:2015a} by Bogfjellmo and Dahmen and
Schmedig, who called it the $GN$-property. It is easy to see that it
is equivalent to being AE.

As we already know, there are a lot of examples of these structures,
like $C^*$-algebras, Banach algebras and more generally: locally
multiplicatively convex algebras. Recall that a locally convex algebra
(either associative or Lie) is called \emph{locally multiplicatively
  convex} if there exists a defining system of seminorms such that for
all seminorms $p$ of this system one has
\begin{equation}
    \label{eq:LMCProperty}
    p(x \cdot y) \le p(x)p(y)
\end{equation}
for all $x, y \in \algebra{A}$. Clearly, this implies
\eqref{Intro:AE}. The associative locally multiplicatively convex
algebras have been discussed in detail by Michael
\cite{michael:1952a}. Clearly, finite-dimensional (Lie)
algebras are Banach and hence locally multiplicatively convex.

Also, all nilpotent locally convex (Lie) algebras belong to this
category, since one just needs to take the maximum of a finite number
of seminorms. So far, we did not find an example for an AE(-Lie)
algebra, which was neither locally multiplicatively convex nor
nilpotent and it seems to be a non-trivial question, whether such an
algebra exists at all.
\begin{question}
    Are there AE(-Lie) algebras, which are not already locally
    multiplicatively convex or nilpotent?
\end{question}

In some special cases, however, we can give an answer to this
question: every associative AE algebra admits an entire holomorphic
calculus in the following sense: let $\algebra{A}$ be a complete
associative AE algebra and let
\begin{equation*}
    f \colon
    \mathbb{C}
    \longrightarrow
    \mathbb{C},
    \quad
    z
    \longmapsto
    f(z)
    =
    \sum\limits_{n = 0}^{\infty}
    a_n z^n
\end{equation*}
be an entire function, then for every $x \in \algebra{A}$ and every
continuous seminorm $p$ with asymptotic estimate $q$ we have
\begin{equation}
    \label{eq:HolomorphicCalculus}
    p \left( f(x) \right)
    =
    p \left(
        \sum\limits_{n = 0}^{\infty}
        a_n x^n
    \right)
    \leq
    \sum\limits_{n = 0}^{\infty}
    |a_n| p \left( x^n \right)
    \leq
    \sum\limits_{n = 0}^{\infty}
    |a_n| q(x)^n
    < \infty.
\end{equation}
Thus $f(x) \in \algebra{A}$ is defined for every $x$ and every entire
function $f$, obeying the usual rules of a functional calculus.  In
this sense, $\algebra{A}$ behaves very much like a locally
multiplicatively convex algebra. If in addition, $\algebra{A}$ is
commutative and Fr\'echet, we get an answer to our question by using a
result \cite{mitiagin.rolewicz.zelazko:1962a} due to Mitiagin,
Rolewicz and \.Zelazko. They showed that an associative, commutative
Fréchet algebra admitting an entire calculus \emph{is} actually
locally multiplicatively convex.

For non-commutative algebras, the situation is different. There are
associative Fr\'echet algebras admitting an entire holomorphic
calculus which are not locally multiplicatively convex. In
\cite{zelazko:1994a}, \.Zelazko gave an example of such an
algebra. However, his example is also not AE.

In the Lie case, not much seems to be known about the relation between
AE-Lie algebras and locally multiplicatively convex Lie algebras.

%
% Another Topology on the Universal Enveloping Algebra
%

\subsection{Another Topology on the Universal Enveloping Algebra}
\label{subsec:AnotherTopologyUniversalEnvelopingAlgebra}

We also imposed a locally convex topology on the universal enveloping
algebra $\mathcal{U}(\lie{g})$, since we can just pull back the
topology on $\Sym_R^{\bullet}(\lie{g})$ with the
Poincar\'e-Birkhoff-Witt isomorphism.  It is not the only one on
$\mathcal{U}(\lie{g})$ which is reasonable: in
\cite{pflaum.schottenloher:1998a} Schottenloher and Pflaum mention
another locally convex topology in the case of a finite-dimensional
Lie algebra, which we will call $\tau$ for now. They take the coarsest
locally convex topology, such that all finite-dimensional
representations of $\lie{g}$ extend to continuous algebra
homomorphisms. This topology is in fact even locally multiplicatively
convex and therefore different from ours. It is easy to show, that our
topology is finer.
\begin{proposition}
    \label{Thm:LCAna:ContinuousRepresentations}%
    Let $R \geq 1$ and $\algebra{U}_R(\lie{g})$ the universal
    enveloping algebra of a finite-dimensional Lie algebra $\lie{g}$,
    then the $\Sym_R$-topology is strictly finer than $\tau$.
\end{proposition}
\begin{proof}
    From Corollary~\ref{Coro:LCAna:ContinuousRepresentations} we know
    that if $\lie{g}$ is finite-dimensional, all finite-dimensional
    representations will extend to continuous algebra
    homomorphisms. Since $\tau$ was the coarsest topology such that
    this is the case, the $\Sym_R$-topology must be finer. It is
    strictly finer, since the $\Sym_R$-topology does not have an
    entire calculus for $R \geq 1$ and can hence not be locally
    multiplicatively convex.
\end{proof}

For representations on infinite-dimensional Banach or Hilbert spaces,
the statement is less important, since there one rarely has
norm-continuous representations, but merely strongly continuous ones.
\begin{remark}
    \label{Rem:LCAnaBCHConvergence}
    One could argue, that a topology which is locally multiplicatively
    convex is much more useful than a ``just locally convex'' one. Of
    course, submultiplicativity of the seminorms allows for more
    constructions like an entire calculus, but there are good reasons
    why the $\Sym_R$-topology also has its advantages: first of all,
    it is also defined for infinite-dimensional Lie algebras with
    reasonable functorial properties. Second, the topology build upon
    representations does not allow to have a continuous deformation of
    $\Sym^{\bullet}(\lie{g})$, at least in a similar way: the
    canonical projections $\pi_n \colon \Sym^{\bullet}(\lie{g})
    \longrightarrow \Sym^n(\lie{g})$ from
    Proposition~\ref{proposition:Projections} can be shown to be
    discontinuous with respect to $\tau$. For $z = 1$, $\xi, \eta \in
    \lie{g}$ and $\pi_1$ being the projection onto the Lie algebra
    itself, we get
    \begin{equation}
        \label{eq:piEinsDiscontinuous}
    	\pi_1 \left(
            \exp(\xi) \star \exp(\eta)
    	\right)
    	=
    	\bch{\xi}{\eta},
    \end{equation}
    if we assume the product and the projection to be convergent. With
    respect to $\tau$, the exponential series $\exp(\xi)$ exists for
    all $\xi \in \lie{g}$ and therefore the Baker-Campbell-Hausdorff
    series must exist for all $\xi, \eta \in \lie{g}$.  Since this is
    not the case in a generic Lie algebra, we get a contradiction.
\end{remark}

%
% Possible Applications and Generalizations
%

\subsection{Possible Applications and Generalizations}
\label{subsec:GeneralizeStuff}

First we note that it is just a matter of collecting signs to extend
our results to super Lie algebras. Here the analysis part is robust
enough as the signs will not alter the estimates. We do not formulate
the details here.

More interesting is the application of our result to deformation
quantization, say in finite dimensions for simplicity. While the Gutt
star product itself may be seen as yet rather simple, it encodes
quantizations of coadjoint orbits: from
\cite{cahen.gutt.rawnsley:1996a} one knows that the Gutt star product
does not just restrict to coadjoint orbits in general, even though
there are interesting exceptional cases. For these cases, we
immediately get convergent star products on coadjoint orbits:
\begin{theorem}
    \label{theorem:QuantizeCoadjointOrbits}%
    Let $G$ be a finite-dimensional Lie group with Lie algebra
    $\lie{g}$ and $\mathcal{O} \subseteq \lie{g}^*$ a coadjoint orbit
    to which the Gutt star product is tangential.  Then the Gutt star
    product restricts to a continuous star product on $\mathcal{O}$.
\end{theorem}
\begin{proof}
    Here we use now that $\Sym_R(\lie{g})$ can be identified with the
    polynomial functions $\Pol(\lie{g}^*)$ on $\lie{g}^*$. Now since
    the evaluation functionals at points in $\lie{g}^*$ are continuous
    by Proposition~\ref{proposition:Projections},
    \refitem{item:PointsArePoints}, the vanishing ideal of the orbit
    is a closed subspace of $\Sym_R(\lie{g})$. Since we assume that
    the Gutt star product is tangential, it is also a two-sided ideal
    with respect to $\star_z$. Thus, on the one hand, the quotient
    inherits the product and, on the other hand, it inherits a
    Hausdorff locally convex topology for which the product is
    continuous. We then can complete again to get a holomorphic
    deformation of those functions on the orbit which are obtained by
    restricting functions on $\lie{g}^*$ in the completion of
    $\Pol^\bullet(\lie{g}^*)$ with respect to the $\Sym_R$-topology to
    the orbit.
\end{proof}

However, even if the Gutt star product is not tangential, there is some
hope to get a star product on coadjoint orbits for which the convergence can
be controlled: in \cite[Thm.~5.2]{bieliavsky.bordemann.gutt.waldmann:2003a}
a construction of a deformed restriction map $\boldsymbol{\iota} = \iota
+ \cdots$ of functions on $\lie{g}^*$ to functions on a coadjoint orbit
$\iota\colon \mathcal{O} \longrightarrow \lie{g}^*$ was given,
provided the Lie algebra $\lie{g}$ is compact and the orbit is
regular. It will be left to a future project to investigate the
behaviour of the deformed restriction map with respect to polynomial
functions and their $\Sym_R$-topology. Beside this fairly general
construction, it seems also plausible to generalize the restriction
procedure to particular other cases.

%
% Appendix
%

\appendix
\section{
	Proofs of the Propositions
	\ref{proposition:PBWGuttDrinfeldBCHWhoElseStarProduct} and
	\ref{proposition:DeformedLieAlgebraStarProduct}
	}
\label{sec:TheProofPBWGDBCH}

Since we used the results from
Proposition~\ref{proposition:PBWGuttDrinfeldBCHWhoElseStarProduct} and
Proposition~\ref{proposition:DeformedLieAlgebraStarProduct},
we need to give a proof that that they hold for any (also
infinite-dimensional) Lie algebra $\lie{g}$. Our analysis is based on
a careful investigation of the BCH series which can be found in the
literature at many places, see e.g. \cite[part 2.8.12]{dixmier:1977a}
or \cite[Chap.~5]{jacobson:1997a}.

To avoid any possible confusion, we define three different star
products.
\begin{definition}
	\label{Def:ThreeStarProducts}
	Let $\lie{g}$ be a Lie algebra and $z \in \mathbb{C}$. Then we define
	the Gutt star product $\star_z$ via \eqref{eq:xikStaretaell}, the
	deformed PBW star product $\widehat{\star}_z$ via
	\eqref{eq:DeformedLieIsGuttStar} and the Drinfel'd star product $\ast_z$
	via 	\eqref{eq:DrinfeldIsGuttStar}.
\end{definition}
We need to prove that all of them are identical. Therefore, it is enough to
show that they coincide on $\xi^n \star \eta$ for $\xi, \eta \in \lie{g}$ and
any $n \in \mathbb{N}$, since those terms generate the whole algebra. We get
equality for $\xi_1 \ldots \xi_n \star \eta$ by polarization and
then for $\xi_1 \ldots \xi_n \star \eta_1 \ldots \eta_{\ell}$ by
iteration.

Recall that we have $\star_z = \widehat{\star}_z$ if $z = 1$ from the
construction, since $\mathfrak{q} = \mathfrak{q}_1$. In
Proposition~\ref{proposition:GuttStarOneLinearFactor} we have sketched the
proof of a formula for $\ast_z$:
\begin{equation}
	\xi^k \ast_z \eta
	=
	\sum\limits_{j=0}^k
    \binom{k}{j} z^j B_j^*
    \xi^{k-j}(\ad_{\xi})^j (\eta).
\end{equation}
We first show $\ast_z = \widehat{\star}_z$.
\begin{proposition}
	\label{proposition:App:DrinfeldIsDeformedStar}
	Let $\lie{g}$ be a Lie algebra, $k \in \mathbb{N}$ and $\xi, \eta \in
	\lie{g}$. Then
	\begin{equation}
		\label{eq:App:DrinfeldIsDeformedStar}
		\xi^k \widehat{\star}_z \eta
		=
		\sum\limits_{j=0}^k
	    \binom{k}{j} z^j B_j^*
	    \xi^{k-j}(\ad_{\xi})^j (\eta).
	\end{equation}
\end{proposition}
\begin{proof}
	We need to show
	\begin{equation*}
		\mathfrak{q}_z \left(
			\sum\limits_{j=0}^k
        	\binom{k}{j} z^j B_j^*
        	\xi^{k-j}(\ad_{\xi})^j (\eta)
		\right)
		=
		\xi^{k} \cdot \eta
	\end{equation*}
	in $\algebra{U}(\lie{g}_z)$. We divide it into two lemmata.
	\begin{lemma}\label{App:Lemma1}
		Let $\xi, \eta \in \mathfrak{g}$ and $k\in \mathbbm{N}$.
		Then we have
		\begin{equation*}
			\mathfrak{q}_z \left(\sum\limits_{n=0}^k z^n
			\binom{k}{n} B_n^* \xi^{k-n}
			\left(\ad_{\xi}\right)^n(\eta)\right)
			=
			\sum\limits_{s=0}^k\mathcal{K}(k,s)
			\xi^{k-s} \cdot \eta \cdot \xi^s
		\end{equation*}
		with
		\begin{equation*}
			 \mathcal{K}(k,s)
			 =
			 \frac{1}{k + 1} \sum\limits_{n=0}^k
			 \binom{k+1}{n} B_n^*
			 \sum\limits_{j=0}^n
			 (-1)^j \binom{n}{j}
			 \sum\limits_{\ell=0}^{k-n} \delta_{s, \ell+j}
		\end{equation*}
	\end{lemma}
	\begin{subproof}
		We need the identities
		\begin{align*}
			\mathfrak{q}_z\left( \xi^k \eta \right)
			& =
			\frac{1}{k + 1}
			\sum\limits_{j = 0}^k
			\xi^{k-l} \cdot \eta \cdot \xi^k
		\intertext{and}
			\mathfrak{q}_z
			\left(
				\left(
					z^n \ad_{\xi}
				\right)^k
				(\eta)
			\right)
			& =
			\sum\limits_{j = 0}^k
			(-1)^j \binom{k}{j}
			\xi^{k-j} \cdot \eta \cdot \xi^j
		\end{align*}
		which are easy to check. Since the map $\mathfrak{q}_z$ is linear,
		we can pull out the constants and get
		\begin{equation*}
			\mathfrak{q}_z \left( \sum\limits_{n=0}^k
			z^n \binom{k}{n} B_n^* \xi^{k-n}
			\left(\ad_{\xi}\right)^k(\eta)\right)
			=
			\sum\limits_{n=0}^k
			\binom{k}{n} B_n^*
			\mathfrak{q}_z\left(z^n \xi^{k-n}
			\left(\ad_{\xi}\right)^k(\eta)\right).
		\end{equation*}
		Now we just have to use the two equalities to get
		\begin{align*}
			\sum\limits_{n = 0}^k
			\binom{k}{n} B_n^*
			\mathfrak{q}_z(z^n \xi^{k-n}
			\left(\ad_{\xi}\right)^k(\eta))
			& =
			\sum\limits_{n = 0}^k
			\binom{k}{n}
			\frac{B_n^*}{k - n + 1}
			\sum\limits_{\ell = 0}^{k - n}
			\xi^{k - n - \ell} \cdot
			\left(
				\sum\limits_{j = 0}^n
				(-1)^j \binom{n}{j}
				\xi^{n-j} \cdot \eta \cdot \xi^j
			\right)
			\cdot \xi^{\ell}
			\\
			& =
			\frac{1}{k + 1}
			\sum\limits_{n = 0}^k
			\binom{k + 1}{n} B_n^*
			\sum\limits_{j = 0}^n
			(-1)^j \binom{n}{j}
			\sum\limits_{\ell = 0}^{k - n}
			\xi^{k - \ell - j}
			\cdot \eta \cdot
			\xi^{\ell + j}
		\end{align*}
		and collect the terms for which we have $\ell + j = s$.
	\end{subproof}
	Now we have to show the following statement:
	\begin{lemma}\label{App:Lemma2}
		Let $\mathcal{K}(k,s)$ be defined as in Lemma \ref{App:Lemma1},
		then we have for all $k\in \mathbbm{N}$
		\begin{equation*}
			\mathcal{K}(k, s)
			=
			\begin{cases}
				1 & s=0 \\
				0 & \text{else}
			\end{cases}
		\end{equation*}
	\end{lemma}
	\begin{subproof}
		We need some standard identities on binomial coefficients and
		the recursive definition of the Bernoulli numbers. Besides that, we
		have to use the equality
		\begin{equation*}
			(-1)^k
			\sum\limits_{j = 0}^k
			\binom{k}{j} B_{m + j}
			=
			(-1)^m
			\sum\limits_{i = 0}^m
			\binom{m}{i} B_{k + i}
		\end{equation*}
		which was proven by Carlitz in \cite{carlitz:1968a}. We take $k \in
		\mathbb{N}$ and prove the lemma by induction over $s$. It is
		straightforward to see that	$\mathcal{K}(k,0) = 1$. The induction
		starts at $s = 1$:
		\begin{align*}
			\mathcal{K}(k,1)
			& =
			\underbrace{
				\frac{1}{k + 1}
			}_{n=0}
			-
			\underbrace{k B_k^*}_{n=k}
			+
			\frac{1}{k + 1}
			\sum\limits_{n = 1}^{k - 1}
			\binom{k + 1}{n} B_n^*
			\left(
				1 + (-1) \binom{n}{1}
			\right)
			\\
			& =
			\underbrace{
				\frac{1}{k+1}
				\sum\limits_{n = 0}^{k - 1}
				\binom{k + 1}{n}
				B_n^*
			}_{ = 1 - B_k^*}
			-
			\frac{1}{k + 1}
			\sum\limits_{n = 0}^k
			\binom{k + 1}{n}
			n B_n^*
			\\
			& =
			1 - B_k^* - k - 1
			+
			\sum\limits_{n = 0}^{k - 1}
			\binom{k}{n} B_n^*
			-
			B_k^*
			\\
			& =
			0.
		\end{align*}
		Now we can do the step $s \rightarrow s + 1$. It is sufficient to
		prove $\mathcal{K}(k,s+1) - \mathcal{K}(k,s) = 0$. For this purpose,
		we rewrite $\mathcal{K}(k,s)$ and we find
		\begin{align*}
			\mathcal{K}(k,s)
			& =
			\frac{1}{k + 1}
			\left[
				1 +
				\sum\limits_{n = k + 1 - s}^s
				\binom{k + 1}{n}
				B_n (-1)^{k - s}
				\binom{n - 1}{k - s}
			\right.
			\\
			& \quad +
			\left.
				\sum\limits_{n = s + 1}^k
				\binom{k + 1}{n} B_n
				\left(
					(-1)^{k - s}
					\binom{n - 1}{k - s}
					+
					(-1)^{n + s}
					\binom{n - 1}{s}
				\right)
			\right].
		\end{align*}
		We get rid of the factor in front by multiplying with $(k + 1)$ and
		finally get after some combinatorial manipulations
		\begin{align*}
			(k+1) \left(
				\mathcal{K}(k,s) - \mathcal{K}(k,s + 1)
			\right)
			& =
			\sum\limits_{n = k - s}^k
			\binom{k + 1}{s + 1}
			\binom{s + 1}{n + s - k}
			B_n (-1)^{k - s}
			\\
			& \quad
			+
			\sum\limits_{n = s + 1}^k
			\binom{k + 1}{s + 1}
			\binom{k - s}{n - s - 1}
			B_n (-1)^{n + s}.
		\end{align*}
		Rewriting it and dividing by $\binom{k + 1}{s + 1}$ (which is not 0
		for $s \leq k$), two sums vanish by the Carlitz identity. We are left
		with
		\begin{equation*}
			(-1)^{s + 1} B_{k + 1}
			-
			(-1)^{k - s} B_{k + 1}
			=
			-(-1)^s
			B_{k + 1}
			\left(
				1 + (-1)^k
			\right)
			=
			0
		\end{equation*}
		since the bracket will be zero if $k$ is odd and $B_{k + 1} = 0$
		if $k$ is even.
	\end{subproof}
	From those two lemmas, the proposition is clear.
\end{proof}
Now, we only have to prove $\star_z = \widehat{\star}_z$ for $z \neq 1$. But we
know that
\begin{equation}
	\xi^k \widehat{\star}_z \eta
	=
	\sum\limits_{j=0}^k
    \binom{k}{j} z^j B_j^*
    \xi^{k-j}(\ad_{\xi})^j (\eta)
\end{equation}
and also
\begin{equation}
	\xi^k \star_1 \eta
	=
	\mathfrak{q}^{-1}
	\left(
		\mathfrak{q}\left( \xi^k \right)
		\cdot
		\mathfrak{q}(\eta)
	\right)
	=
	\sum\limits_{j=0}^k
    \binom{k}{j} B_j^*
    \xi^{k-j}(\ad_{\xi})^j (\eta).
\end{equation}
We finally get from Equation \eqref{eq:xikStaretaell}
\begin{equation}
	\xi^k \star_z \eta
	=
	\sum\limits_{j=0}^k
    \binom{k}{j} z^j B_j^*
    \xi^{k-j}(\ad_{\xi})^j (\eta).
\end{equation}

%
% The bibliographies: dq* are the standard one, but also include all
% the others. The gstar.bib will contain the local additional
% references.
% Make it smaller than other text
%

%{
%  \footnotesize
%  \renewcommand{\arraystretch}{0.5}
%  \bibliographystyle{chairx}
%  \bibliography{dqarticle,dqbook,dqthesis,preprints}
%}

%
% at the very end: the list of fixme's,
% to be deleted in the final version
%

%\ifdraft{\phantomsection}
%\ifdraft{\addcontentsline{toc}{section}{List of Corrections}}
%\ifdraft{\listoffixmes}

%
% end of gstar
%

\end{document}